\documentclass{amsart}
\usepackage{amsfonts,amssymb,amsmath,amsthm}
\usepackage{url}
\usepackage{enumerate}
\usepackage{pifont}
\usepackage[all]{xy}
\usepackage{hyperref}
\usepackage{cleveref}
\usepackage{caption}
\usepackage{geometry}
\newtheorem{theorem}{Theorem}[section]
\newtheorem{lemma}[theorem]{Lemma}
\newtheorem{proposition}[theorem]{Proposition}
\newtheorem{corollary}[theorem]{Corollary}
\theoremstyle{definition}
\newtheorem{definition}[theorem]{Definition}
\newtheorem{remark}[theorem]{Remark}
\newtheorem{notation}{Notation}
\numberwithin{equation}{section}

\global\long\def\C{\mathbb{C}}
\global\long\def\N{\mathbb{N}}
\global\long\def\R{\mathbb{R}}

\global\long\def\mf#1{\mathfrak{#1}}
\global\long\def\mc#1{\mathcal{#1}}

\global\long\def\im{\mathrm{im}}
\global\long\def\id{\mathrm{id}}
\global\long\def\.{,\dots ,}

\global\long\def\Aut{\operatorname{Aut}}
\global\long\def\Lie{\operatorname{Lie}}
\global\long\def\so{\mathfrak{so}}

\global\long\def\ad{\mathrm{ad}}

\global\long\def\so{\mathfrak{so}}

\global\long\def\a{\mathfrak{a}}
\global\long\def\b{\mathfrak{b}}
\global\long\def\f{\mathfrak{f}}
\global\long\def\g{\mathfrak{g}}
\global\long\def\h{\mathfrak{h}}

\global\long\def\l{\mathfrak{l}}
\global\long\def\m{\mathfrak{m}}
\global\long\def\n{\mathfrak{n}}
\global\long\def\p{\mathfrak{p}}
\global\long\def\s{\mathfrak{s}}
\global\long\def\z{\mathfrak{z}}

\global\long\def\Der{\operatorname{Der}}

\author{Reinier Storm}
\address{KU Leuven, Department of Mathematics, Celestijnenlaan 200B -- Box 2400, BE-3001 Leuven, Belgium} 
\email{reinier.storm@kuleuven.be}
\thanks{The author is supported by project 3E160361 of the KU Leuven Research Fund.}


\begin{document}
	
\title{Structure theory of naturally reductive spaces}
\begin{abstract}
	The main result of this paper is that every naturally reductive space can be explicitly constructed from the construction in \cite{Storm2018}. This gives us a general formula for any naturally reductive space and from this we prove reducibility and isomorphism criteria.
\end{abstract}
\maketitle
	
\section{Introduction}
Naturally reductive spaces are amongst the simplest of Riemannian homogeneous spaces. The ones which are Riemannian symmetric are of course the most well known. All isotropy irreducible spaces can also be considered to be naturally reductive.
However, the class of naturally reductive spaces is much broader and contains many other interesting cases.
The holonomy bundle of a naturally reductive connection automatically equips the space with a (non-integrable) $G$-structure, where the naturally reductive connection is also a \emph{characteristic connection} for the $G$-structure. There are many interesting non-integrable $G$-structures one can obtain in this way such as 
homogeneous nearly Kähler manifolds (cf. \cite{Butruille2005, Butruille2010}), homogeneous nearly parallel $G_2$-manifolds (cf. \cite{FriedKathMorSem1997}), 
cocalibrated $G_2$-manifolds (cf. \cite{Friedrich2007}), 
Sasakian $\varphi$-symmetric manifolds (see \cite{Takahashi1977,BlairVanhecke1987_2,BlairVanhecke1987_1}).
Similarly for $Sp(n)Sp(1)$-structures there are the homogeneous $3$-Sasakian manifolds (cf. \cite{BoyerGalickiMann1994}). In \cite{AgricolaFriedrich2010} a connection with parallel skew torsion is constructed for any $7$-dimensional $3$-Sasakian structure. 
Also for $Sp(n)Sp(1)$-structures there are interesting naturally reductive examples. 
One of these is the quaternionic Heisenberg group, which is discussed in \cite{AgricolaFerreiraStorm15}. The naturally reductive connection is here also used to find new examples of generalized Killing spinors. More examples of this phenomena are presented in \cite{AgricolaFerreiraFriedrich2015}.
Naturally reductive spaces have also been used to find new homogeneous Einstein metrics. The simplest examples are the isotropy irreducible spaces, which are necessarily Einstein. D'Atri and Ziller found many other examples of Einstein metrics on naturally reductive compact Lie groups in \cite{D'AtriZiller1979}. Wang and Ziller classified all normal homogeneous Einstein manifolds $G/H$ with $G$ simple in \cite{WangZiller1985}. These metrics are also naturally reductive.
Over the past years there has been an increasing interest in connections with parallel skew torsion because they arise in several fields in theoretical and mathematical physics (e.g. \cite{FriedrichIvanov2002}  and references therein).
The most well known examples of this are naturally reductive connections, which have in particular parallel skew torsion.
The simple geometric and algebraic properties of naturally reductive spaces allow one to classify them in small dimensions. This has been done in \cite{TricerriVanhecke1983,KowalskiVanhecke1983,KowalskiVanhecke1985} in dimension $3,~4,~5$ and more recently in dimension $6$ in \cite{AgricolaFerreiraFriedrich2015}. 

\subsection{Results}
The most important result in this paper is \Cref{thm:canonical base}. This states that any naturally reductive space is in a unique way a $(\mf k,B)$-extension, defined in \cite{Storm2018}, of a space with its transvection algebra of the form
\begin{equation*}
\g=\h\oplus \m_0\oplus_{L.a.} \R^n,
\end{equation*}
where $\h\oplus \m_0$ is semisimple, $\h$ is the isotropy algebra and $\oplus_{L.a.}$ denotes the direct sum of Lie algebras. This implies that the discussion in \cite[Sec.~2.3]{Storm2018} gives an explicit description of all naturally reductive spaces. Remember that $(\mf k,B)$-extensions are particular fiber bundles of naturally reductive spaces. 
More specifically, the fibers are orbits of an abelian group of isometries. This means the fiber distribution is spanned by Killing vectors of constant length, see \cite{Nikonorov2013}. 
Recently in \cite{CleytonMoroianuSemmelmann2018} the authors also investigate in general connections with parallel skew torsion from a fiber bundle perspective. Their approach however does not cover these fiber bundles.
The realization of a naturally reductive space as a $(\mf k,B)$-extension also allows us to prove whether or not it is isomorphic to another naturally reductive space. This is done in \Cref{prop:iso type II}. We also provide easy to check criteria for a naturally reductive space to be irreducible. This is done in the combined results of \Cref{thm:product iff torsion prod}, \Cref{lem:reducible iff ideals} and \Cref{lem:reducibility criteria}. Surprisingly these last two problems were not touched upon in the literature up to now.
It is also nice to note that this approach immediately gives the holonomy algebra of the naturally reductive connection, see \Cref{lem:holonomy of space with g simple} and \eqref{eq:hol alg}. This means we always know what the $G$-structure is which is induced from the holonomy bundle of the naturally reductive connection.

In a forthcoming paper this theory will be used to give a systematic way to classify naturally reductive spaces and explicitly carry this out up to dimension $8$.

\section{preliminaries}
The essential structure of a locally homogeneous space is encoded in the infinitesimal model. We now briefly discuss this below.

\begin{theorem}[Ambrose-Singer, \cite{AmbroseSinger1958}]
A complete simply connected Riemannian manifold $(M,g)$ is a homogeneous Riemannian manifold if and only if there exists a metric connection $\nabla$ with torsion $T$ and curvature $R$ such that 
\begin{equation}
\nabla T = 0 \quad \mbox{and}\quad \nabla R = 0. \label{eq:AS connection}
\end{equation}
\end{theorem}

\begin{remark}
A Riemannian manifold is locally homogeneous if its pseudogroup of local isometries acts transitively on it. It should be noted that there exist locally homogeneous Riemannian manifolds which are not locally isometric to a globally homogeneous space, see \cite{Kowalski1990}. Of course such manifolds have to be non-complete.
\end{remark}

A metric connection satisfying \eqref{eq:AS connection} is called an \emph{Ambrose-Singer connection}. The torsion $T$ and curvature $R$ of an Ambrose-Singer connection evaluated at a point $p\in M$ are linear maps
\begin{equation}
T_p: \Lambda^2 T_p M \to T_p M,\quad  R_p: \Lambda^2 T_p M \to \so(T_p M),\label{eq:(T,R)}
\end{equation}
which satisfy 
\begin{align}
&R_p(x,y)\cdot T_p = R_p(x,y)\cdot R_p=0 \label{eq:par T and R} \\
&\mf{S}^{x,y,z} R_p(x,y)z - T_p(T_p(x,y),z) = 0  \label{eq:B1}\\ 
&\mf{S}^{x,y,z} R_p(T_p(x,y),z) =0  \label{eq:B2},
\end{align}
where $\mf{S}^{x,y,z}$ denotes the cyclic sum over $x,y$ and $z$ and $\cdot$ denotes the natural action of $\so(T_p M)$ on tensors. The first equation encodes that $T$ and $R$ are parallel objects for $\nabla$ and under this condition the first and second Bianchi identity become equations \eqref{eq:B1} and \eqref{eq:B2}, respectively. A pair of tensors $(T,R)$, as in \eqref{eq:(T,R)}, on a vector space $\m$ with a metric $g$ satisfying \eqref{eq:par T and R}, \eqref{eq:B1} and \eqref{eq:B2} is called an \emph{infinitesimal model} on $(\m,g)$. From the infinitesimal model $(T,R)$ of a homogeneous space one can construct a homogeneous space with infinitesimal model $(T,R)$. This construction is known as the \emph{Nomizu construction}, see \cite{Nomizu1954}. This construction goes as follows.
Let 
\[
\h := \{h\in \so(\m) : h\cdot T=0,~ h\cdot R=0\}.
\]
and set
\begin{equation}\label{eq:Nomizu Lie algebra}
\g := \h \oplus \m.
\end{equation}
On $\g$ the following Lie bracket is defined for all $h,k\in \h$ and $x,y\in \m$:
\begin{equation}
[h + x,k + y] := [h,k]_{\so(\m)} - R(x, y) + h(y) - k(x) - T(x,y),\label{eq:Nomizu Lie bracket}
\end{equation}
where $[-,-]_{\so(\m)}$ denotes the Lie bracket in $\so(\m)$. The bracket from \eqref{eq:Nomizu Lie bracket} satisfies the Jacobi identity if and only if $R$ and $T$ satisfy the equations \eqref{eq:par T and R}, \eqref{eq:B1} and \eqref{eq:B2}. We will call $\g$ the \emph{symmetry algebra} of the infinitesimal model $(T,R)$. Let $G$ be the simply connected Lie group with Lie algebra $\g$ and let $H$ be the connected subgroup with Lie algebra $\h$. The infinitesimal model is called \emph{regular} if $H$ is a closed subgroup of $G$. If this is the case, then clearly the canonical connection on $G/H$ has the infinitesimal model $(T,R)$ we started with.
In \cite[Thm.~5.2]{Tricerri1992} it is proved that every infinitesimal model coming from a globally homogeneous Riemannian manifold is regular. 

\subsection{Naturally reductive fiber bundles}
The important results in this paper revolve around the idea of fiber bundles of naturally reductive spaces. We now discuss the basics of this.

\begin{definition}\label{def:nat red decomp}
Let $(\g = \h \oplus \m,g)$ be a Lie algebra together with a subalgebra $\h\subset \g$, a complement $\m$ of $\h$ and a metric $g$ on $\m$. Suppose $\ad(\h)\m\subset \m$ and for all $x,y,z\in \m$ that
\[
g([x,y]_\m,z) = -g(y,[x,z]_\m).
\]
Then we call $(\g=\h\oplus \m,g)$ a \emph{naturally reductive decomposition} with $\h$ the \emph{isotropy algebra}. We will mostly refer to just $\g=\h\oplus \m$ as a naturally reductive decomposition and let the metric be implicit. The infinitesimal model of the naturally reductive decomposition is defined by 
\begin{align}
 T(x,y)&:= -[x,y]_\m,                &\quad \forall x,y\in \m,\label{eq:torsion from lie bracket} \\
  R(x,y)&:=-\ad([x,y]_\h) \in \so(\m), &\quad \forall x,y\in \m,\label{eq:curvature from lie bracket} 
\end{align}
where $[x,y]_\m$ and $[x,y]_\h$ are the $\m$- and $\h$-component of $[x,y]$, respectively. We call the decomposition an \emph{effective} naturally reductive decomposition if the restricted adjoint map $\ad:\h\to \so(\m)$ is injective. We will say that $\g$ is the \emph{transvection algebra} of the naturally reductive decomposition $\g=\h\oplus \m$ if the decomposition is effective and $\im(R) = \ad(\h) \subset \so(\m)$. Note that \eqref{eq:par T and R} implies that $\im(R)\subset \so(\m)$ is a subalgebra and that the transvection algebra is always a Lie subalgebra of the symmetry algebra.
\end{definition}
The proof of the following lemma is straight forward and can be found in \cite{TricerriVanhecke1983}.
\begin{lemma}\label{lem:inf models isomorphic}
Let $(T,R)$ and $(T',R')$ be two infinitesimal models on $(\m,g)$ and $(\m',g')$, respectively. Let $M:\m\to \m'$ be a linear isometry. The following are equivalent
\begin{enumerate}[i)] 
 \item $M\cdot T = T' \quad \mbox{and}\quad M\cdot R = R'$,
 \item the induced map $\hat{M}:\im(R)\oplus \m \to \im(R') \oplus \m'$ is a Lie algebra isomorphism of the transvection algebras.
\end{enumerate}
\end{lemma}

It is important to recognize fiber bundles on the Lie algebra level. The following lemma and definition deal with this and will be used in the sequel.

\begin{lemma}
\label{lem:lie algebra fiber bundle basic lemma}
Let $(\g= \h\oplus \m,g)$ be an effective naturally reductive decomposition. Furthermore, suppose $\m=\m^+\oplus \m^-$ is an orthogonal decomposition of $\h$-modules. Then the following hold:
\begin{enumerate}[i)]
\item $[\m^+ , \m^-]\subset \m$,
\item $[\m^+,\m^-]\subset \m^-$ if and only if $[\m^+,\m^+]_{\m}\subset \m^+$.
\end{enumerate}
If we assume that $[\m^+,\m^-]\subset \m^-$, then also the following hold:
\begin{enumerate}[i)]
\setcounter{enumi}{2}
\item $\b=\h\oplus \m^+$ is a subalgebra of $\g$,
\item $(\g=\b \oplus \m^-,g|_{\m^-\times \m^-})$ is a naturally reductive decomposition.
\end{enumerate}
\end{lemma}
\begin{proof}
\emph{i)} Since $\m^+$ and $\m^-$ are $\h$-invariant we conclude
\[
g(R(u,v)x^+,x^-)=0,\quad \forall~ x^\pm\in \m^\pm,\ \forall u,v\in\m.
\]
Combining this with the fact that $R:\Lambda^2 \m \to \Lambda^2\m$ is symmetric with respect to the Killing form on $\so(\m)\cong \Lambda^2 \m$ it follows that $R(x^+,x^-)=0$ for all $x^\pm\in \m^\pm$. The tensor $R$ is defined by $R(x^+,x^-) = -\ad([x^+,x^-]_\h)$. Since we assume our decomposition to be effective $\ad([x^+,x^-]_\h)=0$ implies that $[x^+,x^-]_\h=0$. Hence $[\m^+ ,\m^-]\subset \m$.

\emph{ii)} Suppose that $[\m^+ ,\m^- ]\subset \m^-$. If $x_1^+,x_2^+\in \m^+$ and $x^-\in \m^-$, then
\[
0=g([x_1^+ ,x^-],x_2^+)=-g(x^-, [x_1^+,x_2^+]).
\]
This implies $[x_1^+,x_2^+]_{\m}\in \m^+$. The converse follows from the same equation and \emph{i)}. 

\emph{iii)} From \emph{ii)} we can easily conclude that $\b$ is a subalgebra of $\g$.

\emph{iv)} For the decomposition $\g=\b\oplus \m^-$ we clearly have $[\b,\m^-]\subset \m^-$ and the decomposition is naturally reductive with respect to the metric $g|_{\m^-\times\m^-}$.
\end{proof}

\begin{definition}\label{def:Lie algebra fiber bundle}
Let $\g=\h\oplus \m$ be a naturally reductive decomposition. Suppose that $[\m^+ ,\m^-]\subset \m^-$, with the notation from Lemma \ref{lem:lie algebra fiber bundle basic lemma}. In this case we will call $\g =\h\oplus \m$ the \emph{decomposition of the total space} of the \emph{infinitesimal fiber bundle} and the naturally reductive decomposition $\g=\b\oplus \m^-$ with isotropy algebra $\b$ the \emph{decomposition of the base space}. 
Furthermore, we will call $\m^+$ the \emph{fiber direction}.
\end{definition}
If connected subgroup $B\subset G$ with $\Lie(B)=\b$ is closed and $G/H$ is globally homogeneous with $H\subset G$ connected, then $G/H\to G/B$ is a homogeneous fiber bundle with $B/H$ as fibers.
In general the Lie group $B$ will not be closed. However, the decomposition $\g=\b\oplus \m^-$ always defines a naturally reductive decomposition and therefore a locally naturally reductive space.  This is the reason why we consider infinitesimal fiber bundles. 

The following is a basic result on tensors which we use in \Cref{lem:split torsion}.

\begin{lemma}
\label{lem:square operation of 2-forms}
Let $(V,g)$ be a finite dimensional vector space with a metric $g$. If $\alpha\in\Lambda^2 V \cong \so(V)$, $\beta\in\Lambda^q V$ and $e_1,\dots,e_n$ and orthonormal basis of $V$, then 
\[
\sum_{i=1}^n(e_i\lrcorner \alpha)\wedge (e_i \lrcorner \beta) = \pi^{\wedge q}(\alpha)\beta \equiv \alpha \cdot \beta,
\]
where $\pi$ is the vector representation of $\so(V)$ and $\pi^{\wedge q}$ is the induced tensor representation on $\Lambda^q V$. Furthermore if $\alpha,\beta\in\Lambda^2 V$, then $\alpha\cdot\beta=[\alpha,\beta]_{\so(V)}$.
\end{lemma}

Next we briefly discuss when a Riemannian manifold with a metric connection which has parallel skew torsion can locally be written as a product. It turns out this only depends on the metric and torsion. This result is essential to prove if a space with parallel skew torsion can not be decomposed as a product.

\begin{lemma}
\label{lem:split torsion}
Let $(V,g)$ be some vector space with a metric $g$. Let $T\in \Lambda^3 V$ be a $3$-form. Let $h\in \mf{so}(V)$ with $h\cdot T=0$. Suppose that either
\begin{enumerate}[i)]
\item $T$ has no kernel and $T=T_{1}+T_{2}\in\Lambda^{3}V_{1}\oplus\Lambda^{3}V_{2}$, with $V_1=(V_2)^\perp$ or,
\item $T$ has a kernel and we set $V_2=\ker(T)$ and $V_1=(V_2)^\perp$, so $T=T_{1}+T_{2}\in\Lambda^{3}V_{1}\oplus\Lambda^{3}V_{2}$ with $T_2=0$.
\end{enumerate}
Then for both cases $h$ leaves $V_1$ and $V_2$ invariant. In other words
\[
\{h\in \so(V):h\cdot T =0\} \cong \{ h_1\in \so(V_1):h_1\cdot T_1=0\}\oplus \{ h_2\in \so(V_2):h_2\cdot T_2=0\}.
\]
\end{lemma}

\begin{proof}
We view $h$ as a skew-symmetric endomorphism of $V$ and we write $h$ as
\[
h=\left(\begin{array}{c|c}
A & -B^{T}\\
\hline B & C
\end{array}\right),
\]
where $A\in\mf{so}(V_1),\ B\in\mbox{Lin}(V_1,V_2),\ C\in\mf{so}(V_2)$. Since the torsion is invariant under $h$ we get
\[
0=h\cdot T=A\cdot T_{1}+B\cdot T_{1}-B^{T}\cdot T_{2}+C\cdot T_{2}.
\]
If any two of these summands are non-zero, then they are linearly independent, since
\[
A\cdot T_{1}  \in  \Lambda^{3}V_{1},\qquad B\cdot T_{1}  \in  \Lambda^{2}V_{1}\otimes V_{2},\qquad -B^{T}T_{2}  \in  V_{1}\otimes\Lambda^{2}V_{2},\qquad C\cdot T_{2}  \in  \Lambda^{3}V_{2}.
\]
Hence all terms vanish. We get 
\[
0=B\cdot T_{1}=(B-B^T)\cdot T_1 = \sum_{i}B(e_{i})\wedge(e_{i}\lrcorner T_1),
\]
where the sum is over an orthonormal basis of $V_{1}$ and $(B-B^T)$ is considered as a block matrix in $\so(V)$. For the last equality we used \Cref{lem:square operation of 2-forms}. The 2-forms $e_{i}\lrcorner T_1$ are all linearly independent, because $T_1$ has no kernel for both case \emph{i)} and case \emph{ii)}. Since $B(e_{i})\in V_{2}$ and $e_i\lrcorner T_1\in \Lambda^2 V_1$ we obtain the equation $B(e_{i})\wedge(e_{i}\lrcorner T_1)=0$ for all $i$. This implies $B(e_{i})=0$ for all $e_{i}$. We conclude that $B=0$ and thus $h$ leaves $V_{1}$ and $V_{2}$ invariant.
\end{proof}

For this reason we make the following definition.

\begin{definition}
\label{def:reducible torsion}
Let $(V,g)$ be some vector space with a metric $g$. A 3-form $T\in \Lambda^3 V$ is called \emph{reducible} if it can be written as $T= T_1 + T_2$ with $T_i\in \Lambda^3 V_i$ for some non-zero $V_1\subset V$ and $V_2\subset V$ such that $V_1\perp V_2$. Otherwise $T$ is called \emph{irreducible}.
\end{definition}

Combining \Cref{lem:split torsion} with de Rham's theorem for Riemannian manifolds we obtain the following.

\begin{theorem}
\label{thm:product iff torsion prod}
Let $(M,g,\nabla)$ be a complete simply connected manifold with a metric connection $\nabla$ with non-zero parallel skew torsion $T$. Then the following are equivalent 
\begin{enumerate}[i)]
\item $M$ is isometric to a product and $\nabla$ is the product connection:
\[
(M,g,\nabla)\cong (M_1,g_1,\nabla_1)\times(M_2,g_2,\nabla_2),
\]
where $\nabla_1$ and $\nabla_2$ are connections on $M_1$ and $M_2$, respectively. Both $\nabla_1$ and $\nabla_2$ have parallel skew torsion.
\item The torsion at some point $x\in M$ is reducible, i.e. $T(x)=T_1(x)+T_2(x)\in\Lambda^{3}V_1(x)\oplus\Lambda^{3}V_2(x)$,
for certain orthogonal subspaces $V_1(x),V_2(x)\subset T_xM$ and $T_i(x)\in \Lambda^3 V_i (x)$. 
\end{enumerate}
\end{theorem}

For naturally reductive spaces this result is already known, see \cite{Tsukada96}. For naturally reductive spaces a criterion on the transvection algebra is more useful.

\begin{definition}
\label{def:irred decomp}
A naturally reductive decomposition $\g=\h\oplus \m$ is \emph{reducible} if its torsion, defined by \eqref{eq:torsion from lie bracket}, is given by $T=T_1+T_2\in \Lambda^3\m_1\oplus \Lambda^3\m_2$, for some non-trivial orthogonal decomposition $\m=\m_1\oplus \m_2$. Otherwise the decomposition is \emph{irreducible}.
\end{definition}

The following classical result due to Kostant (see also \cite{D'AtriZiller1979}) will prove very useful at several points in this paper. 

\begin{theorem}[Kostant,\cite{Kostant1956}] \label{thm:kostant}
Let $(\g=\h\oplus \m,g)$ be an effective naturally reductive decomposition. Then $\mf{k}:= [\m,\m]_{\h}\oplus \m$ is an ideal in $\g$ and there exists a unique $\ad(\mf{k})$-invariant non-degenerate symmetric bilinear form $\overline{g}$ on $\mf{k}$ such that $\overline{g}|_{\m\times\m}=g$ and $[\m,\m]_\h \perp \m$. Conversely, any $\ad(\g)$-invariant non-degenerate symmetric bilinear form on $\g=\h\oplus \m$ with $\m=\h^\perp$ and $\overline{g}|_{\m\times\m}$ positive definite gives a naturally reductive decomposition.
\end{theorem}

Our first reducibility criterion is the following.
\begin{lemma}\label{lem:reducible iff ideals}
	Let $\g=\h\oplus \m$ be a naturally reductive decomposition with $\g$ its transvection algebra. Let $\overline{g}$ be the unique $\ad(\g)$-invariant non-degenerate symmetric bilinear form from Kostant's theorem, see {\rm \Cref{thm:kostant}}. The reductive decomposition $\g=\h\oplus \m$ is reducible if and only if there exist two non-trivial orthogonal ideals $\g_1\subset \g$ and $\g_2\subset \g$ with respect to $\overline{g}$ such that $\g = \g_1\oplus \g_2$ and $\h=\h_1\oplus \h_2$ with $\h_i\subset \g_i$ for $i=1,2$.
\end{lemma}
\begin{proof}
	Assume two such ideals exist. Let $\m_i$ be the orthogonal complement of $\h_i$ inside $\g_i$ for $i=1,2$. Note that $\m_i\neq \{0\}$ for $i=1,2$, because otherwise $\g$ is not the transvection algebra. We clearly have $T\in \Lambda^3\m_1 \oplus \Lambda^3\m_2$, where $T$ is defined by \eqref{eq:torsion from lie bracket}, and the decomposition $\g=\h\oplus \m$ is reducible, see \Cref{def:irred decomp}. 
	
	Conversely suppose that $\g=\h\oplus \m$ is the transvection algebra of a reducible naturally reductive decomposition, i.e. $\m=\m_1\oplus \m_2$ with $\m_1\neq \{0\}$, $\m_2\neq \{0\}$, $\m_1\perp\m_2$, and $[\m_1,\m_2]=\{0\}$. Then
	\[
	\g = [\m,\m]_\h \oplus \m = ([\m_1,\m_1]_\h \oplus \m_1) + ([\m_2,\m_2]_\h\oplus \m_2) = (\h_1\oplus \m_1)+ (\h_2\oplus \m_2),
	\]
	where $\h_i:=[\m_i,\m_i]_\h$. Let $m,m'\in \m_1$ and $n\in \m_2$. Then we have
	\[
	[[m,m']_\h,n]=[[m,m'],n] = [[m,n],m']+[m,[m',n]]=0+0=0.
	\]
	Since elements of the form $[m,m']_\h$ span $\h_1$ it follows that $[\h_1,\m_2] = \{0\}$. In the same way we get $[\h_2,\m_1]=\{0\}$. This also implies that $[\h_1\cap \h_2,\m] =\{0\}$ and because the reductive decomposition is effective we get $\h_1\cap\h_2=\{0\}$. Let $h_1\in \h_1$ and $m_2,m_2'\in \m_2$. Then we have
	\[
	\overline{g}(h_1,[m_2,m_2']_\h) = \overline{g}(h_1,[m_2,m_2'])=\overline{g}([h_1,m_2],m_2')=0.
	\]
	This implies that $\h_1\perp \h_2$ with respect to $\overline{g}$. We conclude that $\g=(\h_1\oplus\m_1)\oplus(\h_2\oplus \m_2)$ is the direct sum of two ideals in the way required.
\end{proof}

\subsection{\texorpdfstring{$(\mathfrak{k},B)$}{(k,B)}-extensions}
Next we briefly recall how a $(\mf k,B)$-extension is defined in \cite{Storm2018}. For a non-zero transvection algebra $\g=\h\oplus \m$ we define a Lie algebra $\s(\g)$ by
\[
\s(\g) = \{f\in \Der(\g):f(\h)=\{0\},~f(\m)\subset \m,~f|_\m\in \so(\m)\}.
\]
If $\g=\{0\}$, then we define $\s(\{0\}) = \so(\infty)$.
For every finite dimensional subalgebra $\mf k\subset \s(\g)$ with an $\ad(\mf k)$-invariant metric $B$ on $\mf k$ we can define a Lie algebra structure on
\[
\g(\mf k):=\h\oplus \mf k\oplus \n\oplus \m,
\]
where $\n\equiv \mf k$ is another copy of $\mf k$. Let $\varphi:\mf k\to \so(\m)$ be the natural Lie algebra representation and let $\psi:\mf k \to \so(\n\oplus\m)$ be the Lie algebra representation $\psi:=\ad\oplus \varphi$. Furthermore, let $(T_0,R_0)$ be the infinitesimal model of $\g=\h\oplus \m$. The Lie bracket on $\g(\mf k)$ is defined by:
\begin{align*}
[h+k,n+m] &= \psi(k)(n+m) + h(m),\quad \forall h\in \h,\forall k\in \mf{k},\forall n\in\n,\forall m\in\m,\\
[h_1+k_1,h_2+k_2] &= [h_1,h_2] + [k_1,k_2] ,\quad \forall h_1,h_2\in \h,\forall k_1,k_2\in \mf{k},\\
[x,y] &= -R_0(x,y) - R_{\mf k}(x,y) - T(x,y)\quad \forall x,y\in \n\oplus \m,
\end{align*}
where we identified $\im(R_0)$ with $\h$, and
\begin{equation}\label{eq: T}
R_{\mf k}(x,y) = \sum_{i=1}^l \psi(k_i)(x,y) k_i,~\quad T=T_0 + \sum_{i=1}^l \varphi(k_i) \wedge n_i + 2 T_{\n},
\end{equation}
and $T_\n (x,y,z) = B([x,y],z)$. Together with the metric $g:=B\oplus g_0$ on $\n\oplus \m$ this defines a naturally reductive decomposition with isotropy algebra $\h\oplus \mf k$, see \cite{Storm2018}. The Lie algebra $\g(\mf k)$ is known as the double extension of $\g$ by $\mf k$, see \cite{MedinaRevoy1985}. The naturally reductive infinitesimal model associated to the decomposition $\g=\h\oplus \mf k\oplus\n\oplus \m$ is $(T,R)$, where $T$ is given by \eqref{eq: T} and $R$ is given by
\begin{equation}\label{eq: R}
R = R_0 + R_{\mf k}.
\end{equation}

\begin{definition}
	We call the infinitesimal model $(T,R)$ the $(\mf k,B)$-\emph{extension} of $(T_0,R_0)$. We also call a naturally reductive decomposition with the infinitesimal model $(T,R)$ the $(\mf k,B)$-extension of the decomposition $\g = \h\oplus \m$. 
\end{definition}
An important property of the Lie algebra $\g(\mf k)$ is that the diagonal $\a \subset \mf k\oplus \n$ is an abelian ideal. The spaces studied in \Cref{sec:type II} are characterized by such ideals. It is interesting to note that every vector in $\mf a$ induces a Killing vector field of constant length on the corresponding homogeneous manifold, see \cite{Nikonorov2013}.

It will be convenient to have the following different formulation of $\s(\g)$, which is used in \Cref{lem:h^+ are derivations}.

\begin{lemma}\label{lem:s(g) as intertwining maps}
Let $\g=\h\oplus \m$ be a naturally reductive decomposition with $\g\neq\{0\}$ its transvection algebra. Let $(T_0,R_0)$ be the infinitesimal model of the decomposition. Let $\so_\h(\m)=\{k\in \so(\m): [k,\ad(h)]_{\so(\m)}=0,~\forall h\in \h\}$.
Then the following holds
\[
\s(\g) \cong \{h\in \so_\h(\m): h\cdot T_0 =0\}.
\]
\end{lemma}
\begin{proof}
For all $k\in \s(\g)$, $h\in \h$ and $m\in \m$ we have
\[
k([h,m]) = [k(h),m] + [h,k(m)] = [h,k(m)].
\]
In other words $\varphi(k) \in \so_\h(\m)$. Furthermore, for all $m_1,m_2\in \m$ we have
\begin{align*}
k(T_0(m_1,m_2))&=-k([m_1,m_2]_\m) = -k([m_1,m_2])  \\
 &= -[k(m_1),m_2)]_\m - [m_1,k(m_2)]_\m =  T_0(k(m_1),m_2) + T_0(m_1,k(m_2)).
\end{align*}
We conclude that $\varphi(k)\cdot T_0 =0$. 

To find a map in the other direction we let $k\in \so_\h(\m)$ with $k\cdot T_0=0$. We define 
\[
\hat{k}:\g \to \g;\quad \hat{k}(h+m) := k(m)
\]
and we show that $\hat{k}\in \s(\g)$. For all $h,h'\in \h$ and $m\in \m$ we have 
\[
\hat{k}([h,h'+m])=\hat{k}([h,h'+m]_\m)=\hat{k}([h,m])=[h,\hat{k}(m)]=[\hat{k}(h),h'+m] + [h,\hat{k}(h'+m)],
\]
where in the before last equality we used $k\in \so_\h(\m)$. It remains to show that for all $m_1,m_2\in \m$ we have
\[
\hat{k}([m_1,m_2]) = [\hat{k}(m_1),m_2]+[m_1,\hat{k}(m_2)].
\]
From $k\cdot T_0=0$ we immediately get
\[
\hat{k}([m_1,m_2])=\hat{k}([m_1,m_2]_\m)= [\hat{k}(m_1),m_2]_\m + [m_1,\hat{k}(m_2)]_\m. 
\]
Furthermore, we have
\begin{align*}
\ad([\hat{k}(m_1),m_2]_\h + [m_1,\hat{k}(m_2)]_\h) &= -R_0(\hat{k}(m_1),m_2) - R_0(m_1,\hat{k}(m_2)) \\
&= -R_0(\hat{k}(m_1)\wedge m_2+ m_1\wedge \hat{k}(m_2))\\
&= -R_0(k\cdot(m_1\wedge m_2)).
\end{align*}
The right-hand-side vanishes precisely when $k\cdot (m_1\wedge m_2)\in \ad(\h)^\perp$, where $\ad(\h)^\perp$ is the orthogonal complement of $\ad(\h)$ in $\so(\m)$ with respect to the Killing form $B_{\so}$ of $\so(\m)$. Note that \Cref{lem:square operation of 2-forms} gives us $k\cdot(m_1\wedge m_2) = [k,m_1\wedge m_2]_{\so(\m)}$. For all $h\in \h$ we have
\[
B_{\so}(\ad(h),[k,m_1\wedge m_2]_{\so(\m)}) =  B_{\so}([\ad(h),k],m_1\wedge m_2) = 0. 
\]
This implies that $R_0(k\cdot (m_1\wedge m_2)) =0$ and thus also $[\hat{k}(m_1),m_2]_\h + [m_1,\hat{k}(m_2)]_\h=0$. From this we now obtain
\[
\hat{k}([m_1,m_2]) =  [\hat{k}(m_1),m_2]_\m + [m_1,\hat{k}(m_2)]_\m=[\hat{k}(m_1),m_2] + [m_1,\hat{k}(m_2)].
\]
Consequently, $\hat{k}$ defines a derivation of $\g$ and $\hat{k}\in \s(\g)$. It is clear that the above two maps are inverse to each other. We conclude that $
\s(\g) \cong \{h\in \so_\h(\m): h\cdot T_0 =0\}$.
\end{proof}

\section{General form of a naturally reductive space\label{sec:general form}}
We define two types of naturally reductive spaces:
\begin{enumerate}[\hspace{20pt} Type I:]
 \item The transvection algebra is semisimple.
 \item The transvection algebra is not semisimple.
\end{enumerate}
First we discuss some basic results for spaces of type I. Most of this section is about describing the spaces of type II. If a Lie algebra is not semisimple, then it contains a non-trivial abelian ideal. This fact will allow us to show that every naturally reductive space of type II is an infinitesimal fiber bundle over another naturally reductive space, see \Cref{def:Lie algebra fiber bundle}. In \Cref{prop:torsion and curvature of type II} we derive a formula for the infinitesimal model of the total space in terms of the infinitesimal model of the base space and a certain Lie algebra representation. This leads us to the main result: 
for every naturally reductive space of type II there exists a unique naturally reductive decomposition of the form $\g=\h\oplus \m \oplus_{L.a.} \R^n$, with $\g$ as its transvection algebra and $\h\oplus \m$ a semisimple algebra, such that the original infinitesimal model of type II is a $(\mf k,B)$-extension of $\g=\h\oplus \m \oplus_{L.a.} \R^n$. Consequently, the construction presented in \cite{Storm2018} generates all naturally reductive spaces.

\subsection{Type I}
In section we will use that there exists for every naturally reductive space a decomposition $\g=\h\oplus \m$ of that space such that the metric on $\m$ is induced by an $\ad(\g)$-invariant non-degenerate symmetric bilinear form on $\g$ for which $\h$ and $\m$ are perpendicular, see \Cref{thm:kostant}.
The results below about spaces of type I are quite elementary. 
The most interesting statement in this section is \Cref{lem:noncompact simple is sym} and the partial duality this induces, see \Cref{def:dual pair} and \Cref{cor:dual compact}.

\begin{lemma}
\label{lem:g simple implies irred}
Let $\g$ be a compact simple Lie algebra together with a negative multiple of its Killing form as $\ad(\g)$-invariant non-degenerate symmetric bilinear form. Any proper subalgebra $\h\subset\g$ gives a reductive decomposition $\g=\h\oplus \m$, with $\m=(\h)^\perp$. This is either an irreducible naturally reductive decomposition with non-zero torsion or the decomposition of an irreducible symmetric space.
\end{lemma}
\begin{proof}
If the torsion is zero, then $\g=\h\oplus \m$ is a decomposition of an irreducible symmetric space. Suppose that the torsion $T$ defined by \eqref{eq:torsion from lie bracket} is non-zero and $T\in \Lambda^3 \m_1\oplus \Lambda^3\m_2$ for some orthogonal decomposition $\m=\m_1\oplus \m_2$. By \Cref{lem:split torsion} the subspace $\h\oplus \m_1$ defines a non-zero ideal of $\g$. Hence it has to be equal to $\g$, which means $\m_1 =\m$. We conclude that $\g=\h\oplus \m$ is irreducible.
\end{proof} 

The next result gives a criterion when $\g$ is the transvection algebra of a reductive decomposition $\g=\h\oplus \m$, with $\g$ semisimple.

\begin{lemma}\label{lem:holonomy of space with g simple}
Let $\g=\h\oplus \m$ be a naturally reductive decomposition with $\g$ semisimple and let $\m \perp \h$ with respect to some $\ad(\g)$-invariant non-degenerate symmetric bilinear form $\overline{g}$ on $\g$ such that $\overline{g}|_{\m\times\m} = g$.
Let $(T,R)$ be the infinitesimal model defined by \eqref{eq:torsion from lie bracket} and \eqref{eq:curvature from lie bracket}. The following hold:
\begin{enumerate}[i)]
 \item if $[\m,\m]_\h=\h$, then $\g$ is the transvection algebra of $(T,R)$,
 \item if $\g$ is simple, then $[\m,\m]_\h=\h$ and by i) the transvection algebra is equal to $\g$,
 \item if the reductive decomposition is effective, then $[\m,\m]_\h=\h$ and $\g$ is the transvection algebra.
\end{enumerate}
\end{lemma}
\begin{proof}
\emph{i)} Let $\ad|_\h:\h\to \so(\m)$ denote the restricted adjoint representation. Let $\mf l:=\ker(\ad|_\h)$. Then $\mf l\subset \h$ is an ideal in $\g$. This ideal is either semisimple or $\{0\}$. Let $\mf l^\perp=\{g\in \g:[g,l]=0,~\forall l\in \mf l\}$ be the complementary ideal. Then $\m\subset \mf l^\perp$ and $[\m,\m]\subset [\mf l^\perp,\mf l^\perp]=\mf l^\perp$. This implies
\[
[\h,\mf l] = [[\m,\m]_\h,\mf l] = [[\m,\m],\mf l] = [[\m,\mf l],\m]+ [\m,[\m,\mf l]] = \{0\}
\]
and thus $\mf l\subset \h\subset \mf l^\perp$. We conclude that $\mf \l=\{0\}$ and $\ad|_\h$ is injective. In particular $\g$ is the transvection algebra.

\emph{ii)} Let $\mf k$ be the subalgebra $\mf k:=[\m,\m]_{\h}\oplus\m$. By Kostant's theorem, \Cref{thm:kostant}, $\mf k$ is a non-zero ideal in $\g$ and thus $\mf k=\g$. This gives us $[\m,\m]_{\h}=\h$ and thus by \emph{i)} the transvection algebra of $(T,R)$ is $\g$.

\emph{iii)} By Kostant's theorem $[\m,\m]_\h\oplus \m$ is an ideal in $\g$. Let $\h_0$ be a complementary ideal. Since $\h_0$ is perpendicular to $\m$ with respect to $\overline{g}$ we have $\h_0\subset \h$ and $[\h_0,\m]=\{0\}$. By assumption we obtain $\h_0=\{0\}$ and thus $[\m,\m]_\h=\h$. Now \emph{i)} implies that $\g$ is the transvection algebra.
\end{proof}

The case that $\g$ is simple and non-compact is very different from the compact case as the following lemma shows.

\begin{lemma}
\label{lem:noncompact simple is sym}
Let $\g$ be a non-compact simple Lie algebra and $\g=\h\oplus \m$ a naturally reductive decomposition. Then $(\g,\h)$ is a symmetric pair.
\end{lemma}

\begin{proof}
By \cite[Thm.~12.1.4]{Wolf2011} we know that any subalgebra $\h$ of a reductive Lie algebra $\g$ is reductive in $\g$ if and only if there is a Cartan involution of $\g$ which stabilizes $\h$, i.e. $\sigma(\h)=\h$. Let $\sigma$ be a Cartan involution which stabilizes $\h$ and let $\h= \h^+ \oplus \h^-$, with 
\[
\h^\pm=\{h\in \h:\sigma(h)=\pm h\}.
\]
The metric on $\m$ is induced from a multiple of the Killing form and $\m=\h^\perp$. The Killing form is invariant under all automorphisms. This implies that $\sigma$ preserves $\m$ as well. Hence we also have $\m=\m^+\oplus \m^-$ with
\[
\m^\pm=\{m\in \m:\sigma(m)=\pm m\}.
\]
Let $\mf g^-=\h^-\oplus \m^-$ and $\mf g^+=\h^+\oplus \m^+$. Since $\sigma$ is a Lie algebra automorphism we immediately get 
\[
[\g^+,\g^+]\subset \g^+,\quad [\g^-,\g^+]\subset \g^-,\quad \mbox{and}\quad [\g^-,\g^-]\subset \g^+.\]
The Killing form is positive definite on $\m^-$ and negative definite on $\m^+$. This implies that either $\m^+=\{0\}$ or $\m^-=\{0\}$.  Suppose that $\m^-=\{0\}$. Then we have
\[
[\h^-,\m]=[\h^-,\m^+]\subset \g^- \cap \m = \{0\}.
\]
This implies that $\h^-\subset \ker(\ad|_\h)$ and by \Cref{lem:holonomy of space with g simple} this implies that $\h^-=\{0\}$. In this case we have $\g^-=\{0\}$ and this contradicts the non-compactness of $\g$. Suppose that $\m^+=\{0\}$. Then we have
\[
[\m,\m]=[\m^-,\m^-]\subset \g^+ = \h^+\subset \h.
\]
This means $(\g,\h)$ is a symmetric pair.
\end{proof}

The above lemma greatly restricts the possible transvection algebras for a type I space. We will now discuss how this allows us to quite easily obtain all type I spaces from the classification of all compact type I spaces. 

\begin{definition}\label{def:nr pair}
	A \emph{naturally reductive pair} $(\g,\h)$ is a Lie algebra $\g$ together with a subalgebra $\h\subset \g$ such that there exists an $\ad(\g)$-invariant non-degenerate symmetric bilinear form $\overline{g}$  for which $\overline{g}|_{\m\times \m}$ is positive definite, where $\m = \h^\perp$ and such that $\g$ is the transvection algebra of the corresponding naturally reductive decomposition.
\end{definition}

\begin{definition}\label{def:dual pair}
	A naturally reductive pair $(\g^*,\h^*)$ is a \emph{partial dual} of a naturally reductive pair $(\g,\h)$ when $\g^*$ is a real form of $\g \otimes \C$ different from $\g$ and the complexified Lie algebra pairs are isomorphic: $(\g\otimes  \C,\h\otimes  \C)\cong (\g^*\otimes \C,\h^*\otimes \C)$.
\end{definition}
First note that the above definition covers the duality of symmetric pairs, with the exception of the self-dual symmetric pair $(\mf{eucl}(\R^n),\so(n))$. We should point out that we are not defining a complete duality for naturally reductive spaces, because it is not a one-to-one correspondence and it is only defined for a very small class of naturally reductive spaces. Also a specific naturally reductive metric does not transfer through the above partial duality.

\begin{corollary}\label{cor:dual compact}
	For every non-compact naturally reductive pair $(\g,\h)$ of type I there exists a partial dual pair $(\g^*,\h^*)$ for which $\g^*$ is compact.
\end{corollary}
\begin{proof}
	Let $\g = \g_1 \oplus_{L.a.} \g_2$ be a direct sum of ideals with $\g_1$ non-compact and simple and suppose for now that $\g_2$ is compact. Let 
	\[
	i=i_1\oplus i_2:\h\to \g_1 \oplus \g_2
	\]
	denote the inclusion of the isotropy algebra. Note that $\n:=i_1(\h)^\perp \subset\g_1$ is non-trivial and contained in $\m=\h^\perp$ for every $\ad(\g)$-invariant non-degenerate symmetric bilinear form. This implies $(\g_1,i_1(\h))$ defines a naturally reductive pair. 
	From \Cref{lem:noncompact simple is sym} it follows that $(\g_1,i_1(\h))\equiv(\g_1,\mf k)$ is a non-compact symmetric pair, where $\mf k\subset \g_1$ is the $+1$ eigenspace of a Cartan involution. We denote the map $i_1$ with restricted codomain by $\varphi:\h\to \mf k$ and the inclusion of $\mf k$ in $\g_1$ by $j:\mf k\to \g_1$. We have $i_1=j\circ \varphi$. Let $(\g_1^*,\mf k)$ be the dual symmetric pair of $(\g_1,\mf k)$ and $j^*:\mf k\to \g_1^*$ the natural inclusion. Let $\h^*:=((j^*\circ \varphi)\oplus i_2)(\h)\subset \g_1^*\oplus \g_2$. 
	It is clear that $(\g^*,\h^*)$ defines a dual naturally reductive pair with $\g^*$ compact.
	If there is more than one non-compact simple factor in $\g$, then we simply apply the above procedure for every factor. 
\end{proof}

\begin{remark}\label{rem:dual spaces type I} 
The process in the above corollary can also be reversed. Let $\g_1$ be compact semisimple and suppose that $(\g_1,i_1(\h))=(\g_1,\mf k)$ is an irreducible compact symmetric pair. Let $(\g_1^*,\mf k)$ be the dual non-compact symmetric pair. Then just as above we obtain a naturally reductive pair $(\g^*:=\g_1^* \oplus \g_2,\h^*)$.
\end{remark}

From \Cref{lem:reducible iff ideals} and the above corollary we see immediately that a non-compact naturally reductive space of type I is irreducible if and only if its compact dual is irreducible.
Dual pairs are algebraically very similar and it is quite easy to obtain all non-compact naturally reductive decompositions from the compact ones because of \Cref{lem:noncompact simple is sym}.

\subsection{Type II \label{sec:type II}}
Now we deal with non-semisimple transvection algebras of naturally reductive spaces.
\begin{lemma}
\label{lem:general form of abelian ideal}
Let $(\g=\h\oplus \m,g)$ be an effective naturally reductive decomposition. Let $\a\subset \g$ be an abelian ideal. Let $\m_\a:=\a\cap \m$ and let $\m_0$ be the orthogonal complement of $\m_\a$ in $\m$. Let $\a':=(\pi_\m|_\a)^{-1} (\m_0)$, were $\pi_\m$ is the projection onto $\m$ along $\h$ in $\g$. Then the following hold:
\begin{enumerate}[i)]
 \item $[\m_\a,\m]=\{0\}$,
 \item $\g':=\h\oplus \m_0$ is a subalgebra of $\g$ and a naturally reductive decomposition,
 \item $\a'$ is an abelian ideal of $\g'$ and $\g$ and satisfies $\a'\cap \m_0 = \a'\cap \h = \{0\}$. 
\end{enumerate}
\end{lemma}
\begin{proof}
\emph{i)} Since the decomposition $\g=\h\oplus \m$ is reductive and $\a$ is an ideal we have
\[
[\h, \m_{\a}]\subset \a \cap \m = \m_{\a}.
\]
Hence $\m_{\a}$ and its orthogonal complement $\m_0$ are $\h$-invariant. Since $\a $ is abelian we have $[\m_{\a} ,\m_{\a}]=\{0\}$.
Let $m\in \m_{\a}$ and $n\in \m_0$. Then we can apply \Cref{lem:lie algebra fiber bundle basic lemma}.\emph{i)} to see that $[m,n]\subset \m$. Combining this with $\a$ being an ideal gives us $[m,n]\in \a\cap \m=\m_{\a}$.
We obtain $g([m,n],[m,n])=-g(n,[m,[m,n]])=0$ and thus $[m,n]=0$. We conclude $[\m,\m_{\a}]=\{0\}$.

\emph{ii)} We already know that $[\h,\m_0]\subset \m_0$. We just saw that $[\m_0,\m_\a]=\{0\} \subset \m_\a$. \Cref{lem:lie algebra fiber bundle basic lemma}.\emph{ii)}  now implies $[\m_0,\m_0]_\m\subset \m_0$. Consequently, $\g'=\h\oplus \m_0$ is a subalgebra and defines a naturally reductive decomposition with respect to the metric $g|_{\m_0\times \m_0}$.

\emph{iii)} We know that $\a'\subset \g'$ and by \emph{ii)} $\g'\subset \g$ is a subalgebra. Hence $[\g',\a']\subset \g'\cap \a = \a'$. This means $\a'$ is an abelian ideal in $\g'$. Clearly $\a'$ is still an abelian ideal in $\g$. Note that $\a'\cap \m_0 \subset \m_\a\cap \m_\a^\perp =\{0\}$. Suppose that $h\in \h\cap \a$. Then for every $n\in \m_0$ we have $[h,n]\in \m_0 \cap \a = \{0\}$.  
If $m\in \m_\a$, then $[h,m]=0$ holds because both $h$ and $m$ are in $\a$ and $\a$ is abelian. By assumption the map $\ad: \h\to \so(\m)$ has trivial kernel. Since $\h\cap \a$ is contained in the kernel we conclude that $\h\cap \a=\{0\}$. In particular $\h\cap \a'=\{0\}$.
\end{proof}

From \Cref{lem:general form of abelian ideal} and \Cref{thm:product iff torsion prod} we immediately obtain that any abelian ideal $\a$ of an irreducible effective naturally reductive decomposition $\h\oplus \m$ satisfies $\a\cap \m = \a\cap \h =\{0\}$. In other words $\m_\a = \{0\}$ if $\h\oplus \m$ is irreducible.

\begin{definition}\label{def:decomposition for abelian ideal}
Let $\g =\h \oplus \m$ be an effective naturally reductive decomposition. Let $\a=\a' \oplus \m_\a$ be as in Lemma \ref{lem:general form of abelian ideal}. Let $\m^+:=\pi_\m(\a')\subset \m_0$, where $\pi_\m:\g\to \m$ is the projection along $\h$ and $\m_0$ is the orthogonal complement of $\m_\a$ in $\m$. Let $\m^-$ be the orthogonal complement of $\m^+$ inside $\m_0$. Furthermore, let $\h^+:=\pi_{\h}(\a')$, where $\pi_\h:\g \to \h$ is the projection along $\m$. Note that $\h^+$ is an ideal in $\h$ because $\pi_\h$ is $\h$-equivariant and $\a$ is an ideal. Let $\h^-$ be a complementary ideal in $\h$, which exists because $\h$ is a reductive Lie algebra. It will be irrelevant which complement we pick. This gives us the following decomposition:
\[
\g = \h^+\oplus \h^- \oplus \m^+ \oplus \m^- \oplus \m_{\a}.
\]
We call this the \emph{fiber decomposition with respect to} $\a$.
\end{definition}

\begin{lemma}\label{lem:brackets m+ m-}
Let the notation be as in Definition {\rm\ref{def:decomposition for abelian ideal}}. Then the following hold:
\begin{enumerate}[i)]
 \item the decomposition $\m=\m^+\oplus\m^-\oplus \m_\a$ is $\h$-invariant,
 \item $[\m^+,\m^+]_\m \subset \m^+$,
 \item $[\h^-,\m^+]=\{0\}$ and $[\h^-,\h^+]=\{0\}$,
 \item $[\a,\m^-\oplus \m_\a]=\{0\}$.
\end{enumerate}
\end{lemma}
\begin{proof}
 \emph{i)} From Lemma \ref{lem:general form of abelian ideal} we know that $\m_\a$ and $\m_0$ are $\h$-invariant. Let $m\in\m^+$ and pick $h\in \h^+$ such that $h+m\in \a'$. Then by \Cref{lem:general form of abelian ideal}.\emph{iii)} we have for every $k\in\h$ the following
\[
\a'\ni[k,h+m]=\underbrace{[k,h]}_{\in\h}+\underbrace{[k,m]}_{\in\m_0}.
\]
Hence $[k,m]\in\m^+$ and thus $\m^+$ is $\h$-invariant. The orthogonal complement $\m^-$ in $\m_0$ is automatically also $\h$-invariant. 

\emph{ii)} Let $m'\in\m^+$ and pick $h'\in \h^+$ such that $h'+m'\in \a$. Then we have
\[
0=[h+m,h'+m']=\underbrace{[h,h']}_{\in\h}+\underbrace{[h,m']}_{\in\m^+}-\underbrace{[h',m]}_{\in\m^+}+[m,m'].
\]
This implies that $[m,m']_{\m}\in\m^+$. 

\emph{iii)} Because $\h^-$ and $\h^+$ are both ideals in $\h$ we get $[\h^-,\h^+]=\{0\}$. Let $h^-\in \h^-$. Then
\[
\a'\ni [h+m,h^-]=[h,h^-]+[m,h^-]=[m,h^-]\in\m^+.
\]
Combining this with $\a'\cap \m^+=\{0\}$ we obtain $[\h^-,\m^+]=\{0\}$.

\emph{iv)} Let $m^-\in \m^-$. Then
\[
\a\ni [h + m, m^-]=\underbrace{[h,m^-]}_{\m^-}+\underbrace{[m,m^-]}_{\m^-}
\]
implies that $[\a,\m^-]\subset \a\cap \m^- = \{0\}$. Since $\a$ is abelian it follows that $[\a,\m_\a]=\{0\}$. 
\end{proof}

In the following we assume we have an abelian ideal $\a\subset \g$ with $\a\cap \m = \a\cap \h=\{0\}$. We let 
\[
\rho:\h^+ \to \m^+
\]
be the linear map defined by the graph $\a \subset \h^+\oplus \m^+$. Let $k\in\h^+$ and $h+m\in\a$. Then
\[
\a\ni[k,h+m]=[k,h]+[k,m].
\]
This implies that $\rho([k,h])=[k,m]=[k,\rho(h)]$, i.e. the linear map $\rho :\h^+\to \m^+$ is an isomorphism of $\h^{+}$-modules. Let $h + m , ~h'+m'\in \a$. Then we have 
\[
0= [h+m,h'+m'] = [h,h'] +[h,m'] + [m,h'] + [m,m'],
\]
or equivalently
\begin{equation}
[h,h']  = -[m,m']_{\h}\qquad\mbox{and}\qquad [m,m']_{\m} = [h',m]+[m',h]=-2\rho([h,h']) \label{eq:formula for rho}.
\end{equation}

\begin{remark}\label{rem:submodule in m^+}
Rewriting \eqref{eq:formula for rho} we get $[m,m']_\m = -2 [h,m']$. This implies that if $\mf v\subset \m^+$ is $\h^+$-invariant, i.e. a submodule, then also $[\mf v,\mf v]_\m \subset \mf v$. 
\end{remark}

Let $\g=\h\oplus \m$ be an effective naturally reductive decomposition which has a non-trivial abelian ideal. If we combine \Cref{lem:lie algebra fiber bundle basic lemma} with \Cref{lem:brackets m+ m-}, then we obtain an infinitesimal fiber bundle, in the sense of \Cref{def:Lie algebra fiber bundle}, for every abelian ideal $\a\subset \g$.

\begin{definition}\label{def:associated base space}
Let $(\g = \h\oplus \m,g)$ be an effective naturally reductive decomposition with a non-trivial abelian ideal $\a\subset \g$ and with infinitesimal model $(T,R)$. Let $\g=\h^+\oplus \h^-\oplus \m^+\oplus \m^-\oplus \m_\a$ be the fiber decomposition with respect to $\a$, see \Cref{def:decomposition for abelian ideal}. Let $\mf e :=\h\oplus \m^+\oplus\m_\a$. \emph{The base space associated to} $\a$ is given by the naturally reductive decomposition
\[
(\mf e\oplus \m^-,g|_{\m^-\times\m^-}),
\]
where $\mf e$ is the isotropy algebra. We will denote the infinitesimal model of the base space, defined by \eqref{eq:torsion from lie bracket} and \eqref{eq:curvature from lie bracket}, by $(T_0,R_0)$.
\end{definition}

\begin{notation}\label{not:phi and psi}
Let $B=\rho^*g|_{\m^+\times \m^+}$ be the pullback metric on $\h^+$. This metric is $\ad(\h^+)$-invariant. We define a 3-form $T_{\h^+}$ on $\h^+$ by $T(h_1,h_2,h_3) := B([h_1,h_2],h_3)$. We define $T_{\m^+}:=\rho(T_{\h^+})$, where $\rho$ is the natural extension $\rho: \Lambda^3\h^+\to\Lambda^3\m^+$. Let 
\[
\varphi:\h^+ \to \so(\m^-)\qquad\mbox{and}\qquad\psi:\h^+ \to \so(\m^+\oplus \m^-)
\]
denote the restricted adjoint representations in $\g$. 
\end{notation}

Note that $T_0$ is invariant under $\varphi(\h^+)$. We now derive a formula for the torsion and curvature of a naturally reductive decomposition $\g=\h\oplus \m$ in terms of $(T_0,R_0)$ and the representations $\varphi$ and $\psi$.

\begin{proposition}
\label{prop:torsion and curvature of type II}
Let $\g = \h\oplus \m$ be an irreducible effective naturally reductive decomposition. Let $\g=\h^+\oplus \h^-\oplus \m^+\oplus \m^-$ be the fiber bundle decomposition associated with an abelian ideal $\a\subset \g$. Its torsion and curvature are given by
\[
T=T_{0}+\sum_{i=1}^{l}\varphi(h_{i})\wedge m_{i}+2 T_{\m^+}\qquad\mbox{and}\qquad  R=R_{0}+\sum_{i=1}^l \psi(h_i)\odot \psi(h_i),
\]
respectively, where $m_1,\dots,m_l$ is an orthonormal basis of $\m^+$ and $h_i:=\rho^{-1}(m_i)$.
\end{proposition}
\begin{proof}
We know by Lemma \ref{lem:brackets m+ m-} that $[\m^+,\m^+]_\m\subset \m^+$. Thus Lemma \ref{lem:lie algebra fiber bundle basic lemma} implies that $[\m^+,\m^-]\subset \m^-$. These two inclusions tell us that
\[
T\in\Lambda^{3}\m^{+}\oplus\Lambda^{2}\m^{-}\otimes\m^{+}\oplus\Lambda^{3}\m^{-}.
\]
The component in $\Lambda^{3}\m^{-}$ is exactly $T_0$ by the definition of $T_0$. Let $h+m\in \a$. Then by Lemma \ref{lem:brackets m+ m-}. \emph{iv)} we have 
\[
0=[h+m,n] = [h,n] + [m,n] = \varphi(h)(n) + [m,n],
\]
for every $n\in \m^-$. This means that $ T(m,n) = - [m,n] = \varphi(h)n$. This proves that the summand in $\Lambda^2 \m^- \otimes \m^+$ is given by $\sum_{i=1}^{l}\varphi(h_{i})\wedge m_{i}$. From \eqref{eq:formula for rho} we know that $[m,m']=-2\rho([h,h'])$. This shows that the summand in $\Lambda^3\m^+$ is given by $2\rho(T_{\h^+})=2T_{\m^+}$.

The curvature of the base space is by definition given by 
\[
R_0(x,y)= - \ad([x,y]_{\mf e }) \in \so(\m^-),\quad \forall x,y\in \m^-.
\]
Let $x,y,u,v\in \m^-$. Then we have
\begin{align*}
R(x,y,u,v) &= g(R(x,y)u,v)=-g([[x,y]_{\h},u],v) \\
           &= -g([[x,y]_{\mf e } - [x,y]_{\m^+},u],v)   \\
           &= R_0(x,y,u,v)+g([[x,y]_{\m^+},u],v)\\
           &= R_0(x,y,u,v)+\sum_{i=1}^l -g(g(\psi(h_i)x,y)[m_i,u],v)\\
           &= R_0(x,y,u,v)+\sum_{i=1}^l g(\psi(h_i)x,y)g(\psi(h_i)u,v)\\
           &= (R_0+\sum_{i=1}^l \psi(h_i) \odot \psi(h_i) )(x,y,u,v).
\end{align*}
Let $x,y\in \m^+$ and $u,v\in \m$.  From \eqref{eq:formula for rho} it follows that
\begin{align*}
[x,y]_\m = \sum_{i=1}^l g([x,y],m_i) m_i = \sum_{i=1}^l g([m_i,x],y)m_i =-2\sum_{i=1}^l g([h_i,x],y)m_i,
\end{align*}
and $[x,y]_\h = \frac{1}{2}\rho^{-1}([x,y]_\m)$. Combining these gives
\[
[x,y]_\h = -\sum_{i=1}^l g([h_i,x],y)h_i.
\]
Consequently,
\[
R(x,y,u,v) = -g([x,y]_{\h}u,v)= \sum_{i=1}^l (\psi(h_i)\odot \psi(h_i))(x,y,u,v).			
\]
From the symmetries of the curvature tensor $R$ we conclude that 
\[
R=R_{0}+\sum_{i=1}^{l}\psi(h_i)\odot \psi(h_i).\qedhere
\]
\end{proof}

In the following lemma we will prove that every effective naturally reductive decomposition admits a maximal abelian ideal. This result will be very useful for the main theorem of this section.

\begin{lemma}\label{lem:max abelian ideal}
Let $\g = \h \oplus \m $ be an effective naturally reductive decomposition. The sum over all abelian ideals inside $\g$ is again an abelian ideal in $\g$. In other words there always exists a maximal abelian ideal. Every derivation of $\g$ preserves the maximal abelian ideal.
\end{lemma}
\begin{proof}
Let $\a:= +_{i} \a_i$ be the sum of all abelian ideals $\a_i$ in $\g$. Then $\a\subset \g$ is an ideal. We have to show for all $x,y\in \a$ that $[x,y]= \sum_{i,j} [x_i,y_j]= 0$, where $x=\sum_i x_i$, $y=\sum_i y_i$, and $x_i,y_i\in \a_i$. 
In other words the sum of two abelian ideals $\a_i$ and $\a_j$ is an abelian ideal in $\g$. It is clear that $\a_i + \a_j$ is an ideal and that $[\a_i,\a_j] \subset \a_i\cap \a_j$. This means that if $\a_{ij}:=\a_i\cap \a_j$ is equal to $\{0\}$, then $\a_i+\a_j$ is also abelian.

Let $\g = \h_i^+\oplus \h_i^-\oplus \m_i^+ \oplus \m_i^-\oplus \m_{\a_i}$ be the fiber decomposition of $\g$ with respect to $\a_i$, see \Cref{def:decomposition for abelian ideal}. The intersection $\a_{ij}=\a_i\cap \a_j$ is an abelian ideal of $\g$ and $\a_{ij}\subset \a_i$. Let $\m^+_{ij}$ be the projection of $\a_{ij}$ onto $\m$. Just as in \Cref{lem:brackets m+ m-}, it follows that $\m^+_{ij}\subset \m_i^+\oplus \m_{\a_i}$ is $\h$-invariant. Let $\mf v_i$ be the orthogonal complement of $\m^+_{ij}$ in $\m^+_i\oplus \m_{\a_i}$. Then $\mf v_i$ is also $\h$-invariant. \Cref{rem:submodule in m^+} implies $[\mf v_i,\mf v_i]_\m\subset \mf v_i$ and $[\m_{ij}^+,\m_{ij}^+]_\m\subset \m_{ij}^+$. Therefore, \Cref{lem:lie algebra fiber bundle basic lemma}.\emph{ii)} implies that $[\mf v_i, \m^+_{ij}]\subset \mf v_i \cap \m_{ij} =\{0\}$. Let $\mf o_i := (\pi_\m|_{\a_i})^{-1}(\mf v_i)$, where $\pi_\m:\g \to \m$ is the projection along $\h$. Then $\mf o_i\subset \a_i$ and thus \Cref{lem:brackets m+ m-} implies that $[\mf o_i, \m_i^-\oplus \m_{\a_i}]=\{0\}$. Since $\mf v_i$ is $\h$-invariant and $\pi_\m$ is $\h$-equivariant we see that also $\mf o_i$ is $\h$-invariant, i.e. $[\h,\mf o_i]\subset \mf o_i$. Finally, we have $[\mf o_i,\a_i] =\{0\}$. In total this tells us that $[\g,\mf o_i] = [\h \oplus \a_i \oplus \m_i^-\oplus \m_{\a_i}, \mf o_i ] \subset \mf o_i$ and thus that $\mf o_i$ is an ideal in $\g$. Moreover $\mf o_i$ is abelian because $\mf o_i\subset \a_i$. We have $\a_i = \mf o_i \oplus \a_{ij}$. By construction we have $\mf o_i\cap \a_j=\{0\}$. This implies that $\mf o_i \oplus \a_j$ is again an abelian ideal. Since $\a_{ij}\subset \a_j$ we obtain
\[
\a_i + \a_j = (\mf o_i \oplus \a_{ij})+\a_j = \mf o_i \oplus \a_j.
\]
We conclude that $\a_i+\a_j$ is an abelian ideal and thus also $\a = +_i \a_i$ is an abelian ideal. Moreover, $\a$ is maximal in the sense that it contains all other abelian ideals.

The maximal abelian ideal of $\g$ is the sum over all abelian ideals. The image of an abelian ideal under an automorphism is an abelian ideal. Therefore, we see that any automorphism preserves the maximal abelian ideal. This implies that also all derivations preserve the maximal abelian ideal.
\end{proof}

\begin{lemma}
\label{lem:hol rep on m^-}
Let $(\g=\h^+\oplus \h^- \oplus \m^+\oplus \m^-,g)$ be an effective naturally reductive decomposition for some abelian ideal $\a\subset \g$ with $\a\cap \m=\{0\}$. Let $\l:= \ker(\varphi)$ and $\l^\perp$ the orthogonal complement in $\h^+$ with respect to $\rho^*g$. Then we have the following decomposition of ideals
\[
\g=(\l \oplus \rho(\l))\oplus_{L.a.} (\l^\perp \oplus \h^- \oplus \rho(\l^\perp) \oplus \m^-).
\]
The restricted representation $\alpha = \ad|_{\l^\perp \oplus \h^-}:\l^\perp \oplus\h^-\to\so(\m^-)$ is faithful. 

\end{lemma}
\begin{proof}
Let $\m^+_\l:=\rho(\l)$ and let $\m^+_{\l^\perp} := \rho(\l^\perp)$ be the orthogonal complement in $\m^+$. Since $\l$ is an ideal we obtain $\m^+_\l \subset \m^+$ is an $\h^+$-invariant subspace and so is $\m^+_{\l^\perp}$. Combining this with \Cref{rem:submodule in m^+} we see that $\l\oplus \m^+_\l$ commutes with $\l^\perp \oplus \m^+_{\l^\perp}$. Let $n\in \m^-$ and $h+m\in \a$ with $h\in \l$, $m\in \m^+_\l$. Then by \Cref{lem:brackets m+ m-} we have
\[
0=[h+m,n]=[h,n]+[m,n]=[m,n].
\]
Hence $\m^+_\l$ also commutes with $\m^-$ and thus it commutes with its orthogonal complement in $\m$. From \Cref{lem:brackets m+ m-}.\emph{iii)} it follows that $\l\oplus \m^+_{\l}$ commutes with $\h^-$. Since $\l\oplus \m^+_{\l}$ is a subalgebra we obtain it is an ideal and it commutes with $\l^\perp\oplus \h^- \oplus \m^+_{\l^\perp} \oplus \m^-$. From \Cref{prop:torsion and curvature of type II} we can immediately see that $\l^\perp \oplus \h^- \oplus \m^+_{\l^\perp} \oplus \m^-$ is a subalgebra and thus also an ideal.

Suppose that $h\in \ker(\alpha)$. For all $m\in \m^+$ and $n\in \m^-$ we have $[m,n]\in \m^-$ by \Cref{lem:brackets m+ m-}.\emph{ii)} and \Cref{lem:lie algebra fiber bundle basic lemma}.\emph{ii)}. Thus
\[
0=[h,[m,n]]=[[h,m],n]+[m,[h,n]]=[[h,m],n].
\]
We conclude that $[h,m]\in \m^+$ commutes with $\m^-$. This implies $\rho^{-1}([h,m])\in \l$. On the other hand $\rho^{-1}([h,m])=[h,\rho^{-1}(m)]\in \mf l^\perp$, because $h\in \mf l^\perp \oplus \h^-$ and $\rho^{-1}(m)\in \h^+$. We obtain $\rho^{-1}([h,m]) \in \l\cap \l^\perp = \{0\}$. Thus $[h,m]=0$ for all $m\in \m^+$. In total we have $[h,\m]=\{0\}$. This implies $h=0$, because we assumed the reductive decomposition to be effective. We conclude $\ker(\alpha)=\{0\}$.
\end{proof}

By \Cref{lem:split torsion} the above \Cref{lem:hol rep on m^-} implies that for an irreducible naturally reductive decomposition $\g=\h\oplus \m$ and any abelian ideal $\a\subset \g$ there are two possible cases: $\ker(\varphi) = \{0\}$ or $\m^-=\{0\}$. The case $\m^-=\{0\}$ corresponds to the $(\mf k,B)$-extensions of a point space.

\begin{lemma} \label{lem:semidirect product}
Let $\g=\h\oplus \m$ be an effective irreducible naturally reductive decomposition with an abelian ideal $\a \subset \g$. Let $\g = \h^+ \oplus \h^- \oplus\m^+ \oplus \m^-$ be the fiber decomposition associated with $\a$. Let $\h_0:=\pi_\h([\m^-,\m^-])$, where this time $\pi_\h$ is the projection onto $\h$ along $\a\oplus\m^-$ in $\g=\h\oplus \a\oplus \m^-$. Let $\h_0^\perp\subset \h $ be a complementary ideal of $\h_0$ in $\h$. Then $\a\oplus \h_0\oplus \m^-$ is a subalgebra of $\g$ and
\[
\g \cong \h_0^\perp \ltimes (\a \oplus \h_0\oplus \m^-).
\]
Moreover, $\a$ is contained in the center of $\a \oplus \h_0\oplus \m^-$. If we define a Lie algebra structure on $\g^-:=\h_0\oplus \m^-$ induced by the quotient $\h_0\oplus \m^-\cong (\a\oplus \h_0\oplus \m^-)/\a$, then $\g^-=\h_0\oplus \m^-$ is a naturally reductive decomposition of the base space, with $\g^-$ its transvection algebra.
\end{lemma}
\begin{proof}
To see that $\a\oplus\h_0\oplus \m^-$ is a subalgebra of $\g$ we first note that $[\h_0,\m^-]\subset \m^-$ and $[\a,\m^-]=\{0\}$, see \Cref{lem:brackets m+ m-}.\emph{iv)}. Therefore, the inclusions which we still need to check are:
\[
 [\m^-,\m^-]\subset \a \oplus \h_0\oplus \m^-  \quad \mbox{and}\quad [\a,\h_0]\subset \a \oplus \h_0\oplus \m^-.
\]
Clearly we have $[\m^-,\m^-]\subset \a \oplus \h_0\oplus \m^-$. We know that $[\a,\m^-]=\{0\}$ and thus 
\[
 [\a,\h_0 ]= [\a,\pi_\h([\m^-,\m^-])]= [\a,[\m^-,\m^-]] = [[\a,\m^-],\m^-]+[\m^-,[\a,\m^-]] =\{0\}.
\]
Thus, $\a\oplus\h_0\oplus \m^-$ is a subalgebra and $\a$ is contained in its center. By definition of $\h_0^\perp$ we have $[\h_0^\perp,\h_0] = \{0\}$. Furthermore, we know $[\h_0^\perp,\a\oplus \m^-]\subset \a\oplus \m^-$. We conclude that $\g\cong \h_0^\perp\ltimes (\a\oplus\h_0\oplus \m^-)$. We have shown that $(\a\oplus \h_0)\oplus \m^-$ is a naturally reductive decomposition of the base space. We also know that $[\a,\m^-]=\{0\}$. Therefore, the quotient $\h_0\oplus \m^-$ still defines a naturally reductive decomposition of the base space. Moreover, this decomposition is effective by \Cref{lem:hol rep on m^-} both for the case $\m^-=\{0\}$ and for the case $\ker(\varphi)=\{0\}$. By definition we have $[\m^-,\m^-]_{\h_0}=\h_0$ and thus we conclude that $\g^-$ is the transvection algebra of the base space.
\end{proof}

\begin{lemma}\label{lem:h^+ are derivations}
Let $\g = \h^+ \oplus \h^- \oplus \m^+ \oplus \m^-$ be an irreducible  naturally reductive decomposition with $\g$ its transvection algebra and with $\ker(\varphi)=\{0\}$. Let $\g^-$ be the Lie algebra from {\rm\Cref{lem:semidirect product}}. Then $\h^+$ can be identified with a subalgebra of $\s(\g^-)$. Moreover, the maximal abelian ideal of $\g^-$ is preserved by $\h^+$.
\end{lemma}
\begin{proof}
By \Cref{lem:semidirect product} we know that $[\a,\h_0]=\{0\}$ and this implies that $[\h^+,\h_0]=\{0\}$. Thus, we obtain
\[
\varphi(\h^+) \subset \{h\in \so_{\h_0}(\m^-):h\cdot T_0=0 \}.
\]
Since $\g^-$ is the transvection algebra of $\h_0\oplus \m^-$ it follows by \Cref{lem:s(g) as intertwining maps} that $\h^+$ is identified with a subalgebra of $\s(\g^-)$. By \Cref{lem:max abelian ideal} all derivations of $\g^-$ preserve the maximal abelian ideal, so in particular $\h^+$ preserves it.
\end{proof}

Let the notation be as in \Cref{lem:semidirect product} and let
\begin{equation}\label{eq:quotient map g to g/a}
p:\g \to \g/\a \cong \h_0^\perp \ltimes \g^-
\end{equation}
be the quotient map. Now we come to the main result of this section.

\begin{theorem}\label{thm:canonical base}
Let $(\g=\h\oplus \m,g)$ be an irreducible naturally reductive decomposition with $\g$ its transvection algebra. Let 
\[
\g = \h^+\oplus \h^-\oplus \m^+\oplus \m^-
\]
be the fiber decomposition with respect to the maximal abelian ideal $\a$. Then the base space associated to $\a$ is isomorphic to the following naturally reductive decomposition
\[
(\g^- = (\h_0 \oplus \m_0 )\oplus_{L.a.} \R^n,g|_{\m^-\times \m^-}),
\]
where $\h_0\oplus \m_0$ is semisimple or $\{0\}$. Moreover, $(\g=\h\oplus \m,g)$ is isomorphic to the $(\varphi(\h^+),\rho^*g|_{\m^+\times \m^+})$-extension of $\g^-=\h_0\oplus \m^-$.
\end{theorem}
\begin{proof}
By assumption our naturally reductive decomposition is irreducible. Therefore, either $\l :=\ker(\varphi)= \h^+$ and $\m^-=\{0\}$ or $\l=\{0\}$ holds by \Cref{lem:hol rep on m^-}. In case $\l=\h^+$ we have $\g^- = \{0\}$ and thus the base space is of the required form. 

Now we consider the case $\l= \ker(\varphi)=\{0\}$. Let $\g^-=\h_0\oplus \m^-$ be the transvection algebra of the base space described in \Cref{lem:semidirect product}. Let $\b$ be the maximal abelian ideal in $\g^-$, which exists by \Cref{lem:max abelian ideal}. Then $\b$ is also an abelian ideal of $\h_0^\perp \ltimes \g^- \cong \g/\a$, where $\h_0^\perp$ is defined in \Cref{lem:semidirect product}.
By \Cref{lem:general form of abelian ideal} we can decompose $\b$ as
\[
\b = \b'\oplus \m^-_{\b},
\]
where $\b'$ satisfies $\b'\cap \m^-=\b'\cap \h_0=\{0\}$. By \Cref{lem:h^+ are derivations} we know that $\h^+\subset \s(\g^-)$ preserves $\b$ and $\m^-$. In particular this tells us that $\m^-_{\b}$ is $\h^+$-invariant and thus also the orthogonal complement of $\m^-_{\b}$ in $\m^-$ is $\h^+$-invariant. This in turn implies that $\b'$ is $\h^+$-invariant. In \Cref{lem:h^+ are derivations} we saw that $[\h^+,\h_0]=\{0\}$. We can write every $b\in \b'$ as $b=h+m^-$ with $h\in \h_0$ and $m^-\in \m^-$. If $d\in \h^+$, then 
\begin{equation}\label{eq:bracket h^+ with (m^-)^+}
\b'\ni d(b) = d(h)+ d(m^-) = d(m^-)\in \m^-.
\end{equation}
Since $\b'\cap \m^-=\{0\}$ we obtain $d(\b' )=\{0\}$. Let $\overline{\a}:=p^{-1}(\b')$, where $p$ is the map from \eqref{eq:quotient map g to g/a}. Then $\overline{\a}$ is an ideal in $\g$ and $\a\subset \overline{\a}$. Note that $p([\overline{\a},\overline{\a}])= [\b',\b']= \{0\}$ and thus $[\overline{\a},\overline{\a}]\subset \a$. Let $\pi_\m:\g \to \m$ be the projection onto $\m$ along $\h$. Then $\pi_\m|_\a$ is injective. This implies that for $x_1,x_2\in \overline{\a}$ we have
\[
[x_1,x_2] = 0 \Leftrightarrow [x_1,x_2]_{\m} = [x_1,x_2]_{\m^+} = 0.
\]
We know that $x_i = a_i+ h_{0,i} + m_i^- $, with $h_{0,i} + m_i^- \in \b'$ and $a_i\in \a$ for $i=1,2$. Let $m_1\.m_l$ be an orthonormal basis of $\m^+$ and $h_j=\rho^{-1}(m_j)$. Then
\[
[x_1,x_2]_{\m^+} = [a_1 + h_{0,1} +m_1^- , a_2 + h_{0,2} +m_2^- ]_{\m^+} = [m_1^-,m_2^-]_{\m^+} = \sum_{j=1}^l g([h_j,m_1^-],m_2^-) m_j,
\] 
where in the second equality we use $[h_{0,1},h_{0,2}]_\m = 0$, $[h_{0,i},m^-_j]\in \m^-$ and that $\a$ commutes with $\h_0\oplus\m^-$. All the summands vanish by \eqref{eq:bracket h^+ with (m^-)^+}. We conclude $[x_1,x_2]=0$ and thus $\overline{\a}$ is an abelian ideal. The maximality of $\a$ implies $\overline{\a}=\a$. Hence $\mf b'=\{0\}$. We have
\[
\g^- = \h_0\oplus \m_0 \oplus \m^-_{\b},
\]
where $\m_0:= (\m^-_{\b})^\perp \subset \m^-$. We know from \Cref{lem:general form of abelian ideal}.\emph{i)} that $[\m_0,\m^-_\b]=\{0\}$. In \Cref{lem:semidirect product} we saw that $\g^-$ is the transvection algebra of $\g^-=\h_0\oplus \m^-$, i.e. $\h_0 = [\m^-,\m^-]_{\h_0}$. Thus, we have 
\begin{align*}
[\h_0,\m^-_{\b}] &= [[\m^-,\m^-]_{\h_0},\m^-_\b] =  [[\m^-,\m^-],\m^-_\b] \\
 &= [[\m^-,\m^-_\b],\m^-] + [\m^-,[\m^-_\b,\m^-]] = \{0\}.
\end{align*}
Hence $\m^-_\b$ is in the center of $\g^-$. By \Cref{lem:general form of abelian ideal}.\emph{ii)} we know that $\h_0\oplus\m_0$ is a subalgebra of $\g^-$. We conclude that
\[
\g^- = (\h_0\oplus \m_0)\oplus_{L.a.}\m^-_\b.
\]
The subalgebra $\h_0\oplus \m_0$ has no non-trivial abelian ideals, since $\b$ is the maximal abelian ideal of $\g^-$. In other words $\h_0\oplus \m_0$ is semisimple or equal to $\{0\}$. The infinitesimal model of the $(\varphi(\h^+),\rho^*g|_{\m^+\times \m^+})$-extension is identified with the infinitesimal model of $\g=\h\oplus \m$ through the isometry $\rho\oplus \id:\h^+\oplus \m^-  \to \m^+\oplus \m^-$. It follows directly from \Cref{prop:torsion and curvature of type II} and the equations \eqref{eq: T} and \eqref{eq: R} that $\rho \oplus \id$ is an isomorphism of the infinitesimal models. We conclude that $(\g=\h\oplus \m,g)$ is isomorphic to the $(\varphi(\h^+),\rho^*g|_{\m^+\times \m^+})$-extension of $(\g^-=\h_0\oplus \m^-, g|_{\m^-\times \m^-})$.
\end{proof}

\begin{definition}
Let the notation be as in \Cref{thm:canonical base}. We call the base space associated with the maximal abelian ideal the \emph{canonical base space}. Furthermore, we will call $\m^+$ the \emph{canonical fiber direction}. 
\end{definition}

\begin{remark}
	The partial duality of pairs from \Cref{def:dual pair} also takes a very simple form for spaces of type II. If two spaces of type II are partial dual to each other, then it easily follows from \Cref{thm:canonical base} that the canonical base spaces also define partial dual pairs. This means that for every naturally reductive pair of type II there exists a partial dual pair for which the semisimple part of the canonical base space is compact.
\end{remark}

\begin{remark}
In \cite{MedinaRevoy1985} the authors proved that the class of Lie algebras which admit an invariant non-degenerate symmetric bilinear form on it is the smallest class which contains the simple and abelian Lie algebras and which is stable under direct sums and double extensions. 
\Cref{thm:canonical base} is similar in the sense that every irreducible infinitesimal model is obtained as a $(\mf k,B)$-extension of an naturally reductive infinitesimal model which has a reductive transvection algebra. The biggest difference is that we do not obtain any new spaces by repeated $(\mf k,B)$-extensions. Therefore the formula in \cite[Sec.~2.3]{Storm2018} directly describes all naturally reductive spaces.
\end{remark}

\section{Isomorphism and irreducibility criteria}

With the knowledge that any naturally reductive decomposition is a particular $(\mf k,B)$-extension we prove in this section relatively easy to check criteria for two naturally reductive spaces to be locally isomorphic. It is also important to known when a naturally reductive space is irreducible. Therefore, we give a necessary and sufficient condition for a $(\mf k,B)$-extension to be irreducible in \Cref{lem:reducibility criteria}. 
First we will investigate under which conditions the canonical base space of a $(\mf k,B)$-extension of some naturally reductive decomposition $\g=\h\oplus \m$ is  again isomorphic to $\g=\h\oplus\m$, which unfortunately is not automatically the case.

Let $(\g=\h \oplus \m,g)$ be a naturally reductive decomposition of the form
\begin{equation}\label{eq:base space form}
\g = \h \oplus \m_0 \oplus_{L.a.} \R^n,
\end{equation}
with $\g$ its transvection algebra and $\h\oplus \m_0$ a semisimple Lie algebra. Let $\g = \g_1\oplus \dots \oplus\g_k\oplus \R^n$, where $\g_1,\dots,\g_k$ are simple ideals of $\g$. Furthermore, let $\mf k\subset \s(\g)$ and $B$ some $\ad(\mf k)$-invariant inner product on $\mf k$. Let $(T,R)$ be the infinitesimal model of the $(\mf k,B)$-extension. The transvection algebra of $(T,R)$ is given by
\begin{equation}\label{eq:extension algebra}
\mf f:= \im(R) \oplus \n \oplus \m,
\end{equation}
with the Lie bracket defined by \eqref{eq:Nomizu Lie bracket}. Let $\mf d\subset \mf f$ be the maximal abelian ideal. We will prove when $\pi_{\n\oplus \m}(\mf d) = \n$, i.e. when the base space $\g = \h \oplus \m$ is equal to the canonical base space of the $(\mf k,B)$-extension. 

\begin{remark}
The map $R|_{\ad(\h\oplus \mf k)}:\ad(\h\oplus\mf k)\to \ad(\h\oplus\mf k)$ is symmetric with respect to the Killing form of $\so(\n\oplus\m)$, denoted by $B_{\so}$, and is given by
\begin{equation}\label{eq:def of Rpsi}
R|_{\ad(\h\oplus\mf k)} = R_0 + \sum_{i=1}^l \psi(k_i) \odot \psi(k_i) ,
\end{equation}
where $k_1,k_2,\dots,k_l$ is an orthonormal basis of $\mf k$ with respect to $B$.  This means we have an orthogonal direct sum 
\[
\ad(\h\oplus\mf k) = \ker(R|_{\ad(\h\oplus\mf k)})\oplus \im(R|_{\ad(\h\oplus\mf k)}).
\]
\end{remark}

\begin{notation}
In this section we will denote $R|_{\ad(\h\oplus \mf k)}$ simply by $R$ and
$R_\psi = \sum_{i=1}^l \psi(k_i) \odot \psi(k_i)$. Furthermore, the center of a Lie algebra $\g$ will be denoted by $\mc{Z}(\g)$ and the semisimple part of a reductive Lie algebra $\g$ will be denoted by $\g^{ss}$. Let $B_{\Lambda^2}$ denote the metric on $\so(\m)$ defined by $B_{\Lambda^2}(x,y) = -\frac{1}{2}\mathrm{tr}(xy)$. Note that $B_{\Lambda^2}$ is a multiple of $B_{\so}$.
\end{notation}

We recall some definitions from \cite{Storm2018}.
\begin{definition}\label{def:k_i and b_i}
Let $(\g=\h \oplus \m,g)$ be a as in \eqref{eq:base space form}. Let $\mf k\subset \s(\g)$ be a Lie subalgebra and let $B$ be an $\ad(\mf k)$-invariant inner product on $\mf k$. Then we define $\varphi_1:\mf k\to\so(\m_0)$ and $\varphi_2:\mf k \to \so(\R^n)$ to be the restricted representations of $\mf k$ and
\[
\mf k_1 := \ker(\varphi_2),\quad \mf k_3 := \ker(\varphi_1),\quad \mf k_2 := (\mf k_1\oplus \mf k_3)^\perp \subset \mf k,
\]
where the orthogonal complement is taken with respect to $B$.
Furthermore, recall that $\mf s(\h\oplus \m_0) \cong \mc{Z}(\h)\oplus \mf p$, where $\mc{Z}(\h)\subset \h$ is the center of $\h$ and $\mf p := \{m\in \m_0:[h,m]=0,~\forall h\in \h\}$. In this way we identify $\mf k_1\oplus \mf k_2\subset \Aut(\h\oplus \m_0)$ with inner derivations: $\b_1\oplus \b_2 \subset \mc{Z}(\h)\oplus \mf p\subset \h\oplus \m_0$.
\end{definition}

\begin{lemma} \label{lem:ker(R) characterisation}
Let $\g=\h\oplus \m$ be a naturally reductive decomposition with $\g$ its transvection algebra as in \eqref{eq:base space form}. Let $\mf k\subset \s(\g)$ and let $B$ be an $\ad(\mf k)$-invariant inner product on $\mf k$. Let $(T,R)$ be the infinitesimal model of the $(\mf k,B)$-extension. 
Then
\[
\ad(\h^{ss})\oplus\ad({\mf k}^{ss})=\ad(\h^{ss}\oplus {\mf k}^{ss})\subset \im(R) \quad\mbox{and}\quad \ker(R) \subset \ad(\mc Z(\h\oplus\mf k)).
\]
Moreover, if $\mf k_1=\{0\}$, then $\ker(R)=\{0\}$. 
\end{lemma}
\begin{proof}
Note that $\h=\h^{ss}\oplus_{L.a.} \mc Z(\h)$. If $h_1,h_2 \in \h^{ss}$ and $k\in \mf k$, then 
\[
B_{\Lambda^2} ( \ad([h_1,h_2]), \psi(k)) = B_{\Lambda^2} ( \ad(h_1),\ad([h_2, k])) =0.
\]
The Lie algebra $\h^{ss}$ is semisimple, so $[\h^{ss},\h^{ss}]=\h^{ss}$. Therefore, for all $h\in\h^{ss}$ and $k\in \mf k$ we obtain $B_{\Lambda^2}(\ad(h),\psi(k)) = 0$. For all $h\in \h^{ss}$ this implies 
\[
R_\psi(\ad(h)) = \sum_{i=1}^l B_{\Lambda^2}(\ad(h),\psi(k_i))\psi(k_i) =0,
\]
where $R_\psi$ is as defined in \eqref{eq:def of Rpsi}. Thus $R(\ad(h)) = R_0(\ad(h)) \neq 0$. 
By assumption we have $R_0(\ad(\h))=\ad(\h)$. Hence, $R_0:\ad(\h)\to\ad(\h)$ is a Lie algebra isomorphism. This implies that $R(\ad(\h^{ss})) = R_0(\ad(\h^{ss})) = \ad(\h^{ss})$. Similarly we prove that $R(\psi({\mf k}^{ss})) = \psi({\mf k}^{ss})$. Consequently, $\ker(R)\subset \ad(\mc Z(\h\oplus \mf k))$, because $\ker(R)\perp \im(R)$ and $\ad(\mc Z(\h\oplus \mf k))$ is the orthogonal complement of $\ad(\h^{ss}\oplus{\mf k}^{ss})$ in $\ad(\h\oplus \mf k)$.

If $\mf k_1=\{0\}$, then $R_0(\ad(\h))\cap R_\psi(\psi(\mf k)) = \{0\}$. Therefore, $\ker(R)= \{0\}$, because $R_0:\ad(\h)\to \ad(\h)$ and $R_\psi:\psi(\mf k)\to\psi(\mf k)$ are both injective.
\end{proof}

Since $\ad(\h\oplus \mf{k})\subset \{h\in \so(\n\oplus\m):h\cdot T=0,~h\cdot R=0\}$ we get a Lie algebra homomorphism
\begin{equation}\label{eq:lie hom g(k) to sym alg}
q:\g(\mf{k})\longrightarrow \ad(\h\oplus \mf{k})\oplus \n\oplus \m;\quad h+k+n+m\mapsto \ad(h+k)+n+m,
\end{equation}
where $\ad(\h\oplus \mf{k})\oplus \n\oplus \m$ is a subalgebra of the symmetry algebra defined in \eqref{eq:Nomizu Lie algebra}. Note that $\mf f$ is a ideal of $\ad(\h\oplus\mf k)\oplus \n\oplus \m$. 
Let $\a \subset \mf k\oplus \n$ be the diagonal subspace. It is easy to see that $\a \subset \g(\mf k)$ is an abelian ideal.
Furthermore, let $p:\g(\mf k)\to \g(\mf k)/\a$ be the quotient Lie algebra homomorphism.
We summarize this in the following diagram:
\begin{equation}\label{eq:diagram p and q}
\xymatrix{
& \g(\mf k)\ar[d]^q\ar[rd]^p &
\\
\mf f~ \ar@{^{(}->}[r]& \ad(\h\oplus \mf k)\oplus \n\oplus \m & \g(\mf k)/\a \cong \mf k\ltimes \g~.
}
\end{equation}
The following proposition proves when the canonical base space of a $(\mf k,B)$-extension of $\g=\h\oplus \m$ is again equal to $\g=\h\oplus \m$. 
In the following the diagonal in ${\mf k}^{ss}\oplus \n^{ss}$ is denoted by $\a^{ss}$. 

\begin{proposition}\label{lem:canonical base iff}
Let $\g=\h\oplus \m$ be as in \eqref{eq:base space form} and let $\mf f$ be the transvection algebra of a $(\mf k,B)$-extension as in \eqref{eq:extension algebra} with maximal abelian ideal $\mf d\subset \mf f$. The following are equivalent
\[
\pi_{\n\oplus \m}(\mf d) = \n \iff \begin{cases}
(i)~~\pi_\m(\mc Z(\b_1))=\{0\}~~ \rm{ and},\\
(ii)~\ker(R)=\{0\},
\end{cases}
\]
where $\pi_{\n\oplus\m}$ and $\pi_{\m}$ denote the projections onto $\n\oplus \m$ and $\m$, respectively.
\end{proposition}
\begin{proof} 
Suppose that $\pi_{\n\oplus \m}(\mf d) \subsetneq \n$. Let $n\in \pi_{\n\oplus\m}(\mf d)^\perp\cap \n$ and $n\neq 0$. Let $k\in\mf k$ be the element corresponding to $n$. From \Cref{lem:ker(R) characterisation} we know that $\psi({\mf k}^{ss})\subset \im(R)$, thus $q(\a^{ss})\subset \mf f$. Note that $\a^{ss}\subset \g(\mf k)$ is an abelian ideal. Thus, the subalgebra  $q(\a^{ss})$ is also an abelian ideal in $\ad(\h\oplus\mf k)\oplus \n\oplus \m$, because $q$ is a surjective Lie algebra homomorphism. Therefore, $q(\a^{ss})\subset \mf d$ and we obtain $\n^{ss}\subset \pi_{\n\oplus\m}(\mf d)$. This implies $k\in \mc Z(\mf k)$. Suppose that $\psi(k)\in \im(R)$. It is easy to see that $k+n\in \mc Z(\g(\mf k))$. The homomorphism $q$ is surjective and thus $q(k+n)=\psi(k)+n\in \mc Z(\mf f)$ and
\[
\mbox{span}\{\psi(k)+n\}\oplus \mf d\subset \mf f
\]
is an abelian ideal. This contradicts the maximality of $\mf d$. We conclude that $\psi(k)\notin \im(R)$ and thus that $\ker(R)\neq \{0\}$. We have shown that $(ii)$ does not hold. Now we can assume that $\n\subset \pi_{\n\oplus \m}(\mf d)$. 
Suppose 
\[
\ad(h'+k')+m \in \mf d,\quad \mbox{with}\quad m\in \m\backslash\{0\}.
\]
We will use the diagram \eqref{eq:diagram p and q} to transfer the abelian ideal $\mf d\subset \mf f$ to $\mf k\ltimes \g$ and conclude that $\pi_\m(\mc Z(\b_1)) \neq \{0\}$.  By \Cref{lem:h^+ are derivations} we know that $\mf d$ is also preserved by all derivations of $\mf f$. As pointed out above, $\mf f\subset \ad(\h\oplus\mf k)\oplus\n\oplus \m$ is an ideal. It follows that $\mf d$ is also an abelian ideal in $\ad(\h\oplus \mf k)\oplus \n\oplus \m$. Note that $\ker(q)\subset \mc Z(\h\oplus \mf k)$ and $\ker(q)$ commutes with $\n\oplus\m$, thus $\ker(q)\subset \mc Z(\g(\mf k))$.
The subspace $q^{-1}(\mf d)$ is a 2-step nilpotent ideal in $\g(\mf k)$ with $\ker(q)$ contained in its center. Therefore, the subalgebra $\tilde{\mf d}:= p(q^{-1}(\mf d))$ is a 2-step nilpotent ideal in $\mf k \ltimes \g$. The reductive decomposition $\ad(\h\oplus\mf k)\oplus \n\oplus \m$ is effective. Thus, we know that $q(\a)+\mf d$ is an abelian ideal in $\ad(\h\oplus\mf k)\oplus \n\oplus \m$, see \Cref{lem:max abelian ideal}. 
Let $\ad(u) + n\in \mf d$ with $u\in \h\oplus \mf k$ and $n\in \n$. Let $k\in \mf k$ such that $k+n\in \a$. We have $\ad(u-k) = \ad(u)+n - (\ad(k)+n) \in \ q(\a)+\mf d$. From \Cref{lem:general form of abelian ideal}. \emph{iii)} we obtain $\ad(u)=\ad(k)$ and thus $q(\a)\subset \mf d$.
In particular, for every $k+n \in \a$ we have
\[
0= [\ad(h'+k')+m,\ad(k)+n] = [\ad(k'),\ad(k)+n)],
\]
where we used \Cref{lem:brackets m+ m-}. 
This implies that $k'\in \mc Z(\mf k)$.
Let 
\[
\tilde{d} := p(h'+k'+m) = k'+h'+m= k' +g_1+\dots + g_k +x \in \tilde{\mf d},
\]
where $x\in \R^n$ and $g_i \in \g_i$ with $\g_i$ a simple ideal of $\g$ for $i=1,\dots,k$. Consider 
\[
[\tilde{d},\g_i] \in \tilde{\mf d}\cap \g_i,
\]
for $i=1,\dots ,k$.
If $[\tilde{d},\g_i]\neq \{0\}$, then this implies that $\g_i \subset \tilde{\mf d}$, because $\g_i$ is simple and $\tilde{\mf d}$ is an ideal. This is not possible because $\tilde{\mf d}$ is 2-step nilpotent and $\g_i$ is simple. We conclude that $[\tilde{d},\g_i]=\{0\}$. Suppose that $y\in \R^n$ and $[k',y]=z \neq 0$. Then $[\tilde{d},y] = z\in \tilde{\mf d}\cap \R^n$. Moreover, $w:=[k',z] \in \tilde{\mf d}\cap \R^n$ and $g(w,y) = g([k',z],y) = -g(z,z)\neq 0$. In particular $w\neq 0$. We already saw that $q(\a) \subset \mf d$. Therefore, $p^{-1}(\tilde{\mf d})=q^{-1}(\mf d)$ and $q(p^{-1}(\tilde{\mf d})) = \mf d$. It follows that $z,w\in \mf d\subset \mf f$. If we take the Lie bracket of $z$ and $w$ in $\mf f$, we obtain
\[
[z,w] = \sum_{i=1}^l g([k_i,z],w) \ad(n_i+k_i) = \sum_{i=1}^l g([k_i,z],[k',z]) \ad(n_i+k_i)\neq 0, 
\]
where $k_1\.k_l$ is an orthonormal basis of $\mf k$ with respect to $B$ and $n_i$ is the corresponding basis of $\n$. This contradicts the fact that $\mf d$ is abelian. We conclude that $[k',y]=0$ for all $y\in \R^n$. In other words $k'\in \mf k_1$. Remember that $k'\in \Der(\g)$ and we showed $k' \in \Der(\g_1\oplus \dots \oplus\g_k)\subset \Der(\g)$ and $k'=-\ad(g_1+\dots +g_k)=-\ad(h'+m)$. 
Hence we see that $0\neq k'\in \mc Z(\mf k_1)$ and $\pi_\m(h'+m) = m \neq 0$. Remember that we defined $\b_1\subset \g_1\oplus \dots \oplus \g_k$ by $\ad(\b_1)=\mf k_1$ in \Cref{def:k_i and b_i}. This proves that $(i)$ does not hold.

For the converse, suppose $\pi_{\n \oplus \m}(\mf d) = \n$. Let $n\in \n$ and $\omega\in \im(R)$ such that $\omega+n\in \mf d$. Let $k\in \mf k$ be the corresponding element of $n$. Note that \Cref{lem:brackets m+ m-}.\emph{iv)} and \Cref{rem:submodule in m^+} imply $\omega = \psi(k)$. Thus for all $k\in \mf k$ we get $\psi(k)\in \im(R)$ and if $\omega'\in \ker(R)$, then $B_{\Lambda^2}(\omega',\psi(k))=0$. It follows that
\[
0=R(\omega') = R_0(\omega') + R_\psi(\omega') = R_0(\omega') + \sum_{i=1}^l B_{\Lambda^2}(\omega',\psi(k_i)) \psi(k_i) = R_0(\omega'),
\]
where $k_1,\dots,k_l$ is an orthonormal basis of $\mf k$. This implies $\omega'\perp \ad(\h)$, because $\im(R_0)=\ad(\h)$ and $R_0$ is symmetric with respect to $B_{\Lambda^2}$. We have $\omega'\perp \ad(\h\oplus\mf k)$ and thus $\omega'=0$. We conclude that $\ker(R) =\{0\}$.

Finally, we still need to show that if $\ker(R)= \{0\}$ and $\pi_\m(\mc Z(\b_1))\neq \{0\}$, then $\pi_{\n\oplus \m}(\mf d)\neq \n$. Let $b=h+m\in \mc Z(\b_1)\subset \h\oplus\m$ with $m\neq 0$. Let $n\in \n$ and $k\in \mf k$ be the elements corresponding to $b$. Since $\ker(R)=\{0\}$ we know that $\psi(k)\in \im(R)$ and $\ad(h)\in \im(R)$. We easily see that $-\psi(k)+\ad(h) + m \in \mc Z(\mf f)$ and thus in particular that $-\psi(k)+\ad(h) + m \in \mf d$. We have $0\neq m\in \pi_{\n\oplus \m}(\mf d)$ and $m\notin \n$ and thus $\pi_{\n\oplus \m}(\mf d)\neq \n$.
\end{proof}

From the above lemma we see that if $\pi_{\n\oplus \m}(\mf d) = \n$, then $\im(R) = \ad(\h\oplus \mf k)$ and $\ad(\mc Z(\mf k_1))\subset \ad(\h)$. More precisely, we obtain
\begin{equation}\label{eq:hol alg}
\im(R) = \ad(\h\oplus\mf k) = \ad(\h)\oplus \psi({\mf k_1}^{ss}\oplus \mf k_2\oplus\mf k_3).
\end{equation}
Note that his is a formula for the holonomy algebra of the naturally reductive connection of a $(\mf k,B)$-extension.

\begin{remark}
	In \cite[Sec.~2.3]{Storm2018} an explicit description of $(\mf k,B)$-extensions of spaces as in \eqref{eq:base space form} is given under the additional assumption that $\b_i\subset \mc{Z}(\h)\oplus \p$ splits as $\b_i =\b_{i,\z}\oplus \b_{i,\p}$ with $\b_{i,\z}\subset \mc{Z}(\h)$ and $\b_{i,\p}\subset \p$ for $i=1,2$. Note that \Cref{lem:canonical base iff}  condition $(i)$ implies this assumption. Consequently, this together with \Cref{thm:canonical base} implies that \cite[Sec.~2.3]{Storm2018} describes really all naturally reductive spaces.
\end{remark}

Next we give a criterion when two $(\mf k,B)$-extensions are isomorphic.

\begin{proposition}\label{prop:iso type II} 
For $i=1,2$ let $\g_i=\h_i\oplus \m_i$ be naturally reductive decompositions with $\g_i$ their transvection algebras and with $\g_i$ of the form
\[
\g_i = \h_i \oplus \m_{0,i} \oplus _{L.a.} \R^{n_i},
\]
where $\h_i\oplus \m_{0,i}$ is semisimple or $\{0\}$. Let $(T_i,R_i)$ be the infinitesimal model of $\g_i=\h_i\oplus \m_i$ for $i=1,2$. Furthermore, let $\f_i = \mf r_i \oplus \n_i \oplus \m_i $ be the transvection algebra of the $(\mf k_i,B_i)$-extension of $(T_i,R_i)$, where $\mf r_i$ is the isotropy algebra. Suppose $\g_i=\h_i\oplus \m_i$ is the canonical base space of the $(\mf k_i,B_i)$-extension for $i=1,2$ and that the $(\mf k_1,B_1)$-extension and $(\mf k_2,B_2)$-extension are isomorphic. Then there is a Lie algebra isomorphism $\tau:\g_1  \to \g_2$.
Furthermore, $\tau(\h_1) = \h_2$, $\tau|_{\m_1}:\m_1\to \m_2$ is an isometry and $\tau_*:\mf k_1\to \mf k_2$ is an isometry, where $\tau_*:\Der(\g_1)\to\Der(\g_2)$ is the induced map on derivations.
\end{proposition}
\begin{proof} 
From \Cref{lem:inf models isomorphic} we obtain a Lie algebra isomorphism
\[
\sigma : \f_1\to \f_2,
\]
such that $\sigma(\mf r_1) = \mf r_2$ and $\sigma$ preserves the unique bilinear form from Kostant's theorem, see \Cref{thm:kostant}. The maximal abelian ideal $\a_1$ of $\f_1$ is bijectively mapped to the maximal abelian ideal $\a_2$ of $\f_2$ by $\sigma$. This implies that $\sigma(\n_1)=\n_2$ and thus we obtain $\sigma(\m_1)=\m_2$, because $\sigma|_{\n_1\oplus \m_1}:\n_1\oplus \m_1\to \n_2\oplus \m_2$ is an isometry. For all $x,y\in \m_1$ we obtain
\[
\sigma(T_1(x,y)) = -\sigma([x,y]_{\m_1}) = -[\sigma(x),\sigma(y)]_{\m_2}=T_2(\sigma(x),\sigma(y))
\]
and
\[
\sigma(R_1(x,y)) = -\sigma(\ad([x,y]_{\mf r_1\oplus \n_1})) = -\ad([\sigma(x),\sigma(y)]_{\mf r_2\oplus \n_2})= R_2(\sigma(x),\sigma(y)),
\]
where $\sigma$ also denotes the linear map $\Lambda^2\m_1\to \Lambda^2\m_2$ induced by $\sigma|_{\m_1}:\m_1\to \m_2$. By \Cref{lem:inf models isomorphic} the isometry $\sigma|_{\m_1}:\m_1\to\m_2$ induces a Lie algebra isomorphism $\tau:\g_1\to\g_2$, which satisfies $\tau(\h_1)=\h_2$ and $\tau|_{\m_1} = \sigma|_{\m_1}$ is an isometry. Recall from \Cref{lem:s(g) as intertwining maps} that $
\s(\g_i) \cong \{x\in \so_{\h_i}(\m_i):h\cdot T_i =0\}$.
Under this identification $\tau_*:\s(\g_1)\to\s(\g_2)$ is given by $\tau_*(x) = \sigma|_{\m_1} \circ x\circ (\sigma|_{\m_1})^{-1}$. Let $k_1\in \mf k_1$ and let $n_1\in \n_1$ be element corresponding to $k_1$. For every $m_2\in \m_2$ we have
\[
(\sigma|_{\m_1} \circ k_1\circ (\sigma|_{\m_1})^{-1})(m_2) = \sigma|_{\m_1}([n_1,(\sigma|_{\m_1})^{-1}(m_2)]) = [\sigma(n_1),m_2].
\]
Remember that by definition $(\mf k_i,B_i)=(\n_i,B_i)$. Therefore, $\tau_*|_{\mf k_1}:\mf k_1\to\mf k_2$ is given by the isometry $\sigma|_{\n_1}:\n_1\to\n_2$. 
\end{proof}

This proposition also implies that the canonical base space is unique for every space. It can be quite non-trivial to see whether two infinitesimal models $(T_1,R_2)$ and $(T_2,R_2)$ on $(\m,g)$ are equivalent. We can view the canonical base space as an invariant of the infinitesimal model. For a base space $\g=\h\oplus \m_0\oplus_{L.a.}\R^n$ it is also quite tractable to decide when two algebras $\mf k_1,\mf k_2\subset \s(\g)$ are conjugate to each other and thus to decide if two naturally reductive spaces are isomorphic.

We are mainly interested in irreducible naturally reductive spaces. This is now  investigated for type II.
Suppose that $\g=\h\oplus \m$ is a naturally reductive decomposition of type II with $\g$ its transvection algebra. Furthermore, suppose that the naturally reductive decomposition is reducible, i.e. 
\[
\g = (\h_1\oplus \m_1) \oplus_{L.a.} (\h_2\oplus \m_2),
\]
see \Cref{lem:reducible iff ideals}. Let $\a\subset \g$ be the maximal abelian ideal. Let $\pi_i:\g \to \g_i:=\h_i\oplus \m_i$ be the projection for $i=1,2$. Now $\pi_i(\a)\subset \g_i$ is an abelian ideal in $\g_i$. Hence $\pi_1(\a)\oplus \pi_2(\a)$ is also an abelian ideal of $\g$. We have $\a\subset \pi_1(\a)\oplus \pi_2(\a)$ and $\a$ is maximal, thus $\a= \pi_1(\a)\oplus \pi_2(\a)$. 
This means that if a reductive decomposition of type II is reducible, then we also obtain a decomposition $\g^- = \g^-_1 \oplus \g^-_2$,
where $\g^-$ is the transvection algebra of the canonical base space as obtained in \Cref{lem:semidirect product}. Moreover, the Lie algebra $\h^+$ splits as an orthogonal direct sum $\h^+=\h^+_1\oplus \h^+_2$, with $\h^+_i\subset \s(\g^-_i)$. Note that it is also possible that $\g_1^- =\{0\}$ or $\g_2^-=\{0\}$. 
Conversely, if we start with a base space
\[
(\g^- = \g^-_1 \oplus \g^-_2,\h_1\oplus \h_2,\overline{g}),
\]
with $\h_i\subset \g_i^-$ and $\g_i^-$ ideals of $\g^-$, a Lie algebra $\mf k=\mf k_1\oplus \mf k_2$ with $\mf k_i\subset \s(\g^-_i)$ and $\mf k_1\perp \mf k_2$ with respect to $B$, then the $(\mf k,B)$-extension is clearly always reducible. 
The above discussion also implies that if a type II space admits a partial dual pair, then it is reducible if and only if its partial dual pair is reducible.

We would also like to have a criterion when a $(\mf k,B)$-extension of a naturally reductive decomposition $\g=\h\oplus \m\oplus_{L.a.} \R^n$ is irreducible. The following proposition will give us such a criterion. For this it is good to remember that the algebra $\s(\h\oplus \m\oplus_{L.a.} \R^n) \cong \mc{Z}(\h) \oplus \mf p \oplus \so(\R^n)$, where $\mf p = \{m\in \m:[h,m]=0,~\forall h\in \h\}$.

\begin{proposition}\label{lem:reducibility criteria}
Let $\g = \h \oplus \m \oplus_{L.a.} \R^n$ be an effective naturally reductive decomposition, with $\g$ its transvection algebra and $\h\oplus \m$ semisimple. Furthermore, let $\mf k\subset \s(\g)$ and let $B$ be some $\ad(\mf k)$-invariant inner product on $\mf k$. Consider the following decomposition
\begin{equation}\label{eq:irreducible factor}
\g=(\h_1\oplus\m_1)\oplus_{L.a.} \dots \oplus_{L.a.} (\h_p \oplus \m_p) \oplus_{L.a.} \m_{p+1} \oplus_{L.a.} \dots \oplus_{L.a.} \m_{p+q},
\end{equation} 
where $\h_i\oplus \m_i$ is an irreducible naturally reductive decomposition with $\h_i\subset \h$ and $\m_i\subset \m$ for $i=1,\dots, p$ and $\m_{p+j}\subset \R^n$ is an irreducible $\mf k$-module for $j=1,\dots, q$. We choose the $\m_1,\dots ,\m_{p+q}$ mutually orthogonal. Suppose that $\g=\h\oplus \m\oplus_{L.a.} \R^n$ is the canonical base space of the $(\mf k,B)$-extension. The $(\mf k,B)$-extension is reducible if and only if there exists a non-trivial partition:
\[
\{\m_1,\dots,\m_p,\m_{p+1},\dots,\m_{p+q}\}=W'\cup W'', \quad W'\cap W'' = \emptyset,
\]
and an orthogonal decomposition of ideals $\mf k=\mf k'\oplus \mf k''$ with respect to $B$ such that $\mf k'$ acts trivially on all elements of $W''$ and $\mf k''$ acts trivially on all elements of $W'$.
\end{proposition}
\begin{proof}
If such a partition exists, then it is clear from the formula of the $(\mf k,B)$-extension and \Cref{thm:product iff torsion prod} that the $(\mf k,B)$-extension is reducible. 

For the converse we suppose the $(\mf k,B)$-extension is reducible. Let $\mf v:= \{v\in \m\oplus \R^n: \varphi(k)v=0,~\forall k\in \mf k\}$. Suppose that $\m_i\subset \mf v$ for some $i=1\. p+q$. Then we can define a partition by $W':=\{\m_i\}$, $W'':=\{\m_1\.\hat{\m}_i\.\m_{p+q}\}$ and define $\mf k':=\{0\}$ and $\mf k'':=\mf k$. From now on we assume that no $\m_i$ contained in $\mf v$.
Let $\f$ be the transvection algebra of the $(\mf k,B)$-extension $(T,R)$. If the $(\mf k,B)$-extension is reducible, then by \Cref{lem:reducible iff ideals} there exist two orthogonal ideals $\f_1\subset \f$ and $\f_2\subset \f$ with respect to the unique bilinear form from Kostant's theorem, such that $\f=\f_1\oplus \f_2$ and $\im(R)=\mf r_1\oplus \mf r_2$ with $\mf r_i\subset \f_i$. Let $\a\subset \f$ be the maximal abelian ideal. Let $\pi_i:\f \to \f_i$ be the projections for $i=1,2$. Now $\pi_i(\a)\subset \f_i$ is an abelian ideal in $\f_i$. Hence also $\pi_1(\a)\oplus \pi_2(\a)$ is an abelian ideal of $\f$. Since $\a\subset \pi_1(\a)\oplus \pi_2(\a)$ and $\a$ is maximal we obtain $\a= \pi_1(\a)\oplus \pi_2(\a)$. Hence $\n = \n' \oplus \n''$ with $\n'\subset \f_1$ and $\n''\subset \f_2$. In particular this implies that $\n'\perp\n''$. Let $\mf k=\mf k'\oplus \mf k''$ be the corresponding orthogonal decomposition of $\mf k$. We will now show for all $\m_i$ that either $\m_i\subset \f_1$ or $\m_i\subset \f_2$.
Since there is no $\m_i$ contained in $\mf v$ we have
\[
\R^n= [\mf k,\R^n] = [\mf k',\R^n]+ [\mf k'',\R^n].
\]
Note that $[\mf k',\R^n]\subset \f_1$ and $[\mf k'',\R^n]\subset \f_2$, hence $\R^n=[\mf k',\R^n]\oplus [\mf k'',\R^n]$. This implies that $\m_{p+j}$ is contained in either $\f_1$ or $\f_2$ for all $j=1,\dots,q$.
We consider the case that $\h_i\oplus \m_i$ is not a reductive decomposition of an irreducible symmetric space. Note that $[\mf k,\m_i]\neq \{0\}$, because we assumed that $\m_i$ is not contained in $\mf v$. Suppose that $v\in [\mf k',\m_i]$ for some $v\neq 0$ and $1\leq i\leq p$. Then $v\in \f_1\cap \m_i$. Define $V_0:=\{v\}$ and $V_j:=\mbox{span}\{V_{j-1},~[V_{j-1},\m_i]_{\m_i}\}$ for $j\geq 1$. By assumption $\h_i\oplus \m_i$ is an irreducible decomposition. It is easy to see that this implies there exists a $p\in \N$ for which $V_p = \m_i$. Since $\f_1$ is an ideal we conclude that $\m_i \subset \f_1$. Similarly with $\mf k'$ replaced by $\mf k''$ and $\f_1$ replaced by $\f_2$. If $\h_i\oplus \m_i$ defines an irreducible symmetric space, then $\s(\h_i\oplus \m_i) = \mc Z(\h_i)$. If $\mc Z(\h_i)=\{0\}$, then $\m_i\subset \mf v$ and this we assumed not to be the case. The irreducible symmetric spaces for which $\mc Z(\h_i)\neq 0$ are exactly the irreducible hermitian symmetric spaces and $\mc Z(\h_i)$ is then 1-dimensional. If $z\in \mc Z(\h_i)\backslash \{0\}$, then $\ad(z)$ is a multiple of the almost complex structure on $\m_i$, see \cite[Ch.~VIII]{Helgason2001}. Thus, $[z,\m_i]=\m_i$ holds. By assumption $\varphi(\mf k)$ does not act trivially on $\m_i$, so there either is some $k'\in \mf k'$ which acts on $\m_i$ by the derivation $\ad(z)$ or otherwise there is some $k''\in \mf k''$ which acts on $\m_i$ by the derivation $\ad(z)$. In this first case we have $\m_i = [z,\m_i] = [k',\m_i]\subset \f_1$. In the second case we have $\m_i = [k'',\m_i]\subset \f_2$. 
This shows that either $\m_i$ is contained in $\f_1$ or that $\m_i$ is contained in $\f_2$. We can define a partition by $\m_i\in W'$ if $\m_i\subset \f_1$ and $\m_i \in W''$ if $\m_i\subset \f_2$. Then $\mf k'$ acts trivially on all elements of $W''$ and $\mf k''$ acts trivially on all elements of $W'$.
\end{proof}

\Cref{thm:canonical base} together with the results in \cite{Storm2018} give us an explicit construction for any naturally reductive space. Furthermore, we showed in this section that this general formula of a naturally reductive space allows us to decide when two naturally reductive spaces are isomorphic or whether one naturally reductive space is irreducible. In a forthcoming paper we will illustrate the use of these results by classifying all naturally reductive spaces up to dimension 8.

\medskip

\proof[Acknowledgement] This paper is part of my PhD thesis supervised by Professor Ilka Agricola, whom I would like to thank for her ongoing support and guidance.

\bibliographystyle{plain_url}
\bibliography{./NaturalReductiveDim7}

\end{document}